\documentclass[11pt,reqno,oneside]{amsart}

\usepackage{graphicx}
\usepackage{amssymb}
\usepackage{epstopdf}

\usepackage[a4paper, total={5.56in, 9in}]{geometry}

\usepackage{amsmath,amsfonts,amsthm,mathrsfs,amssymb,cite}
\usepackage[usenames]{color}

\usepackage{xcolor}

\newtheorem{thm}{Theorem}[section]

\newtheorem{cor}[thm]{Corollary}
\newtheorem{lem}{Lemma}[section]
\newtheorem{prop}{Proposition}[section]
\theoremstyle{definition}
\newtheorem{defn}{Definition}[section]

\theoremstyle{remark}

\newtheorem{rem}{Remark}[section]
\numberwithin{equation}{section}

\numberwithin{equation}{section}

\newcounter{saveeqn}

\allowdisplaybreaks

\title[Geometric properties of transmission eigenfunctions and applications]{Local geometric properties of conductive transmission eigenfunctions and applications}

\author{Huaian Diao}
\address{School of Mathematics, Jilin University, Changchun 130012, China}
\email{diao@jlu.edu.cn}

\author{Xiaoxu Fei}
\address{School of Mathematics and Statistics, Northeast Normal University, Changchun 130024, China}
\email{feixx548@nenu.edu.cn}

\author{Hongyu Liu}
\address{Department of Mathematics, City University of Hong Kong, Kowloon, Hong Kong SAR, China}
\email{hongyu.liuip@gmail.com, hongyliu@cityu.edu.hk}

\date{} 

\begin{document}
\maketitle

	\begin{abstract}
	
The purpose of the paper is twofold. First, we show that partial-data transmission eigenfunctions associated with a conductive boundary condition vanish locally around a polyhedral or conic corner in $\mathbb{R}^n$, $n=2,3$. Second, we apply the spectral property to the geometrical inverse scattering problem of determining the shape as well as its boundary impedance parameter of a conductive scatterer, independent of its medium content, by a single far-field measurement. We establish several new unique recovery results. The results extend the relevant ones in \cite{DCL} in two directions: first, we consider a more general geometric setup where both polyhedral and conic corners are investigated, whereas in \cite{DCL} only polyhedral corners are concerned; second, we significantly relax the regularity assumptions in \cite{DCL} which is particularly useful for the geometrical inverse problem mentioned above. We develop novel technical strategies to achieve these new results.

	\medskip
	
	
	\end{abstract}
	
\section{Introduction}

\subsection{Mathematical setup and summary of major findings}

The purpose of the paper is twofold. We are concerned with the spectral geometry of transmission eigenfunctions and the geometrical inverse scattering problem of recovering the shape of an anomalous scatterer, independent of its medium content, by a single far-field measurement. We first introduce the mathematical setup of our study.

Let $\Omega$ be a bounded Lipschitz domain in $\mathbb{R}^n$, $n=2,3$, with a connected complement $\mathbb{R}^n\backslash\overline{\Omega}$. Let $V\in L^\infty(\Omega)$ and $\eta\in L^\infty(\partial\Omega)$ be complex-valued functions. Let $\Gamma$ denote an open subset of $\partial\Omega$. Consider the following conductive transmission eigenvalue problem associated with $k\in \mathbb R_+$ and $(w,v)\in H^{1}(\Omega)\times H^1(\Omega)$:
\begin{equation}\label{0eq:tr}
   \begin{cases}
 \ \  \Delta w+k^{2}(1+V)w=0  &\quad \mbox {$\mathrm {in}\ \  \Omega$},\\
  \ \ \Delta v+k^{2}v=0   &\quad \mbox {$\mathrm {in}\ \ \Omega$},\\
   \ \ w=v,\ \partial_{\nu}w=\partial_{\nu}v+\eta v  &\quad \mbox{$\mathrm {on}\ \  \Gamma$},
   \end{cases}
   \end{equation}
 where and also in what follows, $\nu\in\mathbb{S}^{n-1}$ signifies the exterior unit normal vector to $\partial\Omega$. Clearly, $(w, v)=(0, 0)$ is a trivial solution to \eqref{0eq:tr}. If there exists a nontrivial pair of solutions to \eqref{0eq:tr}, $k$ is referred to as a conductive transmission eigenvalue and $(u, v)$ is the corresponding pair of conductive transmission eigenfunctions. In the case $\Gamma=\partial\Omega$, \eqref{0eq:tr} is said to be the full-data conductive transmission eigenvalue problem, and otherwise it is called the partial-data problem. $\eta$ is called the boundary impedance or conductive parameter. If $\eta\equiv 0$, then \eqref{0eq:tr} is reduced to the standard transmission eigenvalue problem. Hence, the conductive transmission eigenvalue problem \eqref{0eq:tr} is a generalized formulation of the transmission eigenvalue problem. Nevertheless, it has its own physical background when $\eta\equiv\hspace*{-3.5mm}\backslash\, 0$ as shall be discussed in what follows.

One of the main purposes of this paper is to quantitatively characterize the geometric property of the partial-data conductive transmission eigenfunctions (assuming their existence). The major findings can be briefly summarized as follows. If there is a polyhedral or conic corner on $\partial\Omega$, then under certain regularity conditions the eigenfunctions must vanish at the corner. The regularity conditions are characterized by the H\"older continuity of the parameters $q:=1+V$ and $\eta$ locally around the corner as well as a certain Herglotz extension property of the eigenfunction $v$, which is weaker than the H\"older continuity. The results extend the relevant ones in \cite{DCL} in two directions: first, we consider a more general geometric setup where both polyhedral and conic corners are investigated, whereas in \cite{DCL} only polygonal  and edge corners are concerned; second, we significantly relax the regularity assumptions in \cite{DCL} which is particularly useful for the geometrical inverse problem discussed in what follows. We develop novel technical strategies to achieve those new results. More detailed discussion shall be given in the next subsection.

The other focus of our study is the inverse scattering problem from a conductive medium scatterer. Let $V$ be extended by setting $V=0$ in $\mathbb{R}^n\backslash\overline{\Omega}$. Throughout, we set $q=1+V$. Let $u^i(\mathbf{x})$ be a time-harmonic incident wave which is an entire solution to
\begin{align}\label{eq:incident}
	\Delta u^i(\mathbf x)+k^2 u^i(\mathbf x)=0,\quad \quad \mathbf x\in \mathbb R^n,
\end{align}
where $k\in \mathbb R_+$ signifies the wave number. Let $(\Omega, q, \eta)$ denote a conductive medium scatterer with $\Omega$ signifying its shape and $q, \eta$ being its medium parameters. The impingement of $u^i$ on $(\Omega, q, \eta)$ generates wave scattering and it is described by the following system:
\begin{equation}\label{eq:contr}
\begin{cases}
\Delta u^-+k^2qu^-=0,  &\quad \mbox{in}\quad \Omega,\\
\Delta u^++k^2u^+=0,   &\quad \mbox{in}\quad \mathbb R^n\setminus \overline{\Omega},\\
u^+=u^-,\quad \partial_\nu u^++\eta u^+=\partial _\nu u^-, &\quad \mbox{on}\quad \partial \Omega,\\
u^+=u^i+u^s, &\quad\mbox{in}\quad \mathbb R^n\backslash\overline{\Omega},\\
\lim_{r\to \infty}r^{(n-1)/2}(\partial _ru^s-\mathrm iku^s)=0, &\quad r=\vert \mathbf x\vert,
\end{cases}
\end{equation}	
where $\mathrm i:=\sqrt{-1}$ and the last limit in \eqref{eq:contr} is known as the Sommerfeld radiation condition that characterises the outward radiating of the scattered wave field $u^s$. The well-posedness of the direct problem \eqref{eq:contr} can be found in \cite{BO} for the unique existence of $u:=u^-\chi_\Omega+u^+\chi_{\mathbb{R}^n\backslash\overline{\Omega}} \in H^1_{\rm loc}(\mathbb R^n)$.
 Moreover, the scattered field admits the following asymptotic expansion:
\begin{equation}\notag
u^s(\mathbf x)=\frac{e^{ik\vert \mathbf x\vert }}{\vert \mathbf x\vert^{(n-1)/2} }\left(u^\infty(\hat{\mathbf x})+\mathcal O\left(\frac{1}{\vert \mathbf x\vert^{(n-1)/2}}\right)\right), \quad \vert \mathbf x\vert \to \infty,
\end{equation}
which holds uniformly in all directions  $\hat{\mathbf{x}}:={\mathbf x}/\vert \mathbf x\vert\in\mathbb{S}^{n-1}$. The function $u^\infty$ defined on the unit sphere $\mathbb S^{n-1}$ is known as the far field pattern of $u^s$. Associated with \eqref{eq:contr}, we are concerned with the following geometrical inverse problem:
\begin{equation}\label{5eq:ineta}
u^\infty(\hat{\mathbf x};u^i),\ u^i\ \mbox{fixed}\longrightarrow \Omega\quad \mbox{independent of $q$ and $\eta$.}
\end{equation}
That is, we intend to recover the geometrical shape of the conductive scatterer independent of its physical content by the associated far-field pattern generated by a single incident wave (which is usually referred to as a single far-field measurement in the literature).

 Determining the shape of a scatterer from a single far-field measurement constitutes a longstanding problem in the inverse scattering theory \cite{DR2018,DR,Liu22}. In this paper, based on the spectral geometric results discussed earlier, we derive several new unique identifiability results for the inverse problem \eqref{5eq:ineta}. In brief, we establish local unique recovery results by showing that if two conductive scatterers possess the same far-field pattern, then their difference cannot possess a polyhedral or conic corner. If we further imposed a certain a-priori global convexity on the scatterer, then one can establish the global uniqueness result. Moreover, we can show that the boundary impedance parameter $\eta$ can also be uniquely recovered. It is emphasized that all of the results established in this paper hold equally for the case $\eta\equiv 0$. If $\eta\equiv 0$, \eqref{eq:contr} describes the scattering from a regular medium scatterer $(\Omega, q)$. In the case $\eta\neq 0$, $(\Omega, q, \eta)$ (effectively) characterises a regular medium scatterer $(\Omega, q)$ by a thin layer of highly loss medium \cite{Ang,BO,CDL2020}, and in two dimensions \eqref{eq:contr} describes the corresponding transverse electromagnetic scattering, whereas in three dimensions \eqref{eq:contr} describes the corresponding acoustic scattering. In addition to its physical significance, introducing a boundary parameter $\eta$ make our study more general which includes $\eta\equiv 0$ as a special case. Hence, in what follows, we also call $(v, w)$ to \eqref{0eq:tr} as generalized transmission eigenfunctions.
 

 \subsection{Connection to existing studies and discussions}

 Before discussing the relevant existing studies, we note one intriguing connection between the scattering problem \eqref{eq:contr} and the spectral problem \eqref{0eq:tr}. If $u_\infty\equiv 0$, which by Rellich's theorem implies that $u^+=u^i$ in $\mathbb{R}^n\backslash\overline{\Omega}$, one can show that $(v, w)=(u^i|_{\Omega}, u^-|_{\Omega})$ fulfils the spectral system \eqref{0eq:tr} with $\Gamma=\partial\Omega$. In the case of $u_\infty\equiv 0$, no scattering pattern can be observed outside $\Omega$, and hence the scatterer $(\Omega, q, \eta)$ is invisible/transparent with respect to the exterior observation under the wave interrogation by $u^i$. On the other hand, if $(w, v)$ is a pair of full-data transmission eigenfunctions to \eqref{0eq:tr}, then by the Herglotz extension $v$ can give rise to an incident wave whose impingement on $(\Omega, q, \eta)$ is (nearly) no-scattering, i.e. $(\Omega, q, \eta)$ is (nearly) invisible/transparent.

Recently, there has been considerable interest in quantitatively characterising the singularities of scattering waves induced by the geometric singularities on the shape of the underlying scatterer as well as its implications to invisibility and geometrical inverse problems. There are two perspectives in the literature. The first one is mainly concerned with occurrence or non-occurrence of non-scattering phenomenon, namely whether invisibility can occur or not. The main rationale is that if the scatterer possesses a geometric singularity (in a proper sense) on its shape, then it scatters a generic incident wave nontrivially, namely invisibility cannot occur. Here, the generic condition is usually characterized by a non-vanishing property of the incident wave at the geometrically singular place. It first started from the study in \cite{BPS} for acoustic scattering with many subsequent developments in different physical contexts \cite{LME,B2018,BL,Blasten2020,CX21,CV,DCL,ElH,SS,VX,BLY,BLX2020,2021,LX,DFLY,BL2018}. The other one is a spectral perspective which is mainly concerned with the spectral geometry of transmission eigenfunctions. According to the connection mentioned above, the spectral geometric results characterise the patterns of the wave propagation inside a (nearly) invisible/transparent scatterer. It was first discovered in \cite{BL2017} that transmission eigenfunctions are generically vanishing around a corner point and such a local geometric property was further extended to conductive transmission eigenfunctions in \cite{DCL}, elastic transmission eigenfunctions in\cite{BL,DLS}  and electromagnetic  transmission eigenfunctions in \cite{2021,DFLY,BLX2020}. Though the two perspectives share some similarities, especially about the vanishing of the wave fields around the geometrically singular places, there are subtle and technical differences. In fact, it is numerically observed in \cite{BLLW} that there exist transmission eigenfunctions which do not vanish, instead localize, around geometrically singular places. An unobjectionable reason to account for such (locally) localizing behaviour of the transmission eigenfunctions is the regularity of the eigenfunctions at the geometrically singular places. In general, if the transmission eigenfunctions are H\"older continuous, they locally vanish around the singular places. Nevertheless, it is shown in \cite{LT} that under a certain Herglotz extension property, the locally vanishing property still hold. It is shown in \cite{LT} that the aforementioned regularity criterion in terms of the Herglotz extension is weaker than the H\"older regularity. In addition to the local geometric pattern, the spectral geometric perspective also leads to the discovery of certain global geometric patterns of the transmission eigenfunctions. Indeed, it is discovered in \cite{CDHLW,DJLZ,DLX1} that the (full-data) transmission eigenfunctions tend to (globally) localize on $\partial\Omega$ with many subtle structures. Those spectral geometric results have been proposed to produce a variety of interesting applications, including super-resolution imaging \cite{CDHLW}, artificial mirage \cite{DLX1} and pseudo plasmon resonance \cite{ACL}. We also refer to \cite{Liu22} for more related results in different physical contexts.

In this paper, we adopt the second perspective to study the (local) geometric properties of the conductive transmission eigenfunctions as well as consider the application to address the unique identifiability issue for the geometrical inverse scattering problem. As discussed in the previous subsection, our results derived in this paper extend the relevant ones in \cite{DCL} in terms of the geometric setup as well as the regularity requirements. To achieve these new results, we develop novel technical strategies. In principle, we adopt microlocal tools to quantitatively characterise the singularities of the eigenfunctions induced by the corner or conic singularities. Nevertheless, we utilise CGO (Complex Geometric Optics) solutions of the PDO (partial differential operator) $\Delta+(1+V)$ in our quantitative analysis, whereas in \cite{DCL}, the analysis made use of certain CGO solutions to $\Delta$. This induces various subtle and technical quantitative estimates and asymptotic analysis. Finally, as also discussed in the previous subsection, we apply the newly derived spectral geometric results to establish several novel unique identifiability results for the geometric inverse problem \eqref{5eq:ineta}. We would also like to mention in passing some recent results on determining the shape of a scattering object by a single or at most a few far-field measurements in different physical contexts \cite{DLS,DLW,DLZZ21,CDL2,CDL3,DLW20,BL2018,BLX2020,DFLY,B2018,CDLZ22,BL,LPRX,Liu-Zou3}.

The rest of the paper is organized as follows. In Section \ref{sec:preliminary}, we collect some preliminary results which are needed in the subsequent analysis. In Section \ref{sec:2D}, we show that the conductive transmission eigenfunctions to \eqref{0eq:tr} near a convex sectorial corner in $\mathbb R^2$ must vanish. In Section \ref{sec:3D}, we study the vanishing of conductive transmission eigenfunctions to  \eqref{0eq:tr}  near a convex conic or polyhedral corner in $\mathbb R^3$. In Section \ref{sec:inverse}, we  discuss the visibility of a scatterer associated  with \eqref{eq:contr}. Furthermore, the unique recovery for the shape determination $\Omega$ associated with the corresponding conductive scattering problem \eqref{eq:contr} is investigated.


   \section{Preliminaries}\label{sec:preliminary}

   In this section, we present some preliminary results which shall be frequently used in our subsequent analysis.

   Given $s\in \mathbb{R}$ and $p\geq 1$,  the Bessel potential space is defined  by
   \begin{equation}\label{1eq:Hsp}
   H^{s,p}:=\{f\in L^{p}(\mathbb{R}^{n}); \mathcal{F}^{-1}[(1+\vert \xi \vert^{2})^{\frac{s}{2}}\mathcal{F}f]\in L^{p}(\mathbb{R}^{n}) \},
   \end{equation}
   where $\mathcal{F}$ and $\mathcal{F}^{-1}$ denote the Fourier transform and its inverse, respectively.


   	
 The following proposition on a multiplication property  for Sobolev spaces  can be directly proved by utilizing the results in \cite[Theorem 7.5]{AM}, \cite[Proposition 7.6]{Blasten2020}, and \cite[Proposition 3.1]{CX21}.
   	
   \begin{prop}\label{1pro:nor}
Suppose that $q \in H^{1,1+\epsilon_{0}}$, where $ 0<\epsilon_{0}<1$. It holds that
   	\begin{equation}\notag
   	\vert \vert qf\vert \vert _{H^{1,\tilde{p}}} \leq C\vert \vert f\vert \vert_{H^{1,p}}\quad \mbox{for any}\quad  f\in H^{1,p} \mbox{ and } p\geq 1,
   	\end{equation}
   	where $C$ is a positive constant and  $1<\tilde{p}<2$ satisfies
   	\begin{equation}\label{eq:p til cond}
   		\frac{1}{p}+\frac{1}{1+\epsilon_{0}}=\frac{1}{\tilde{p}}\quad \mbox{and}
   		\quad \frac{1}{n+1}+\frac{1}{p}\leq \frac{1}{\tilde{p}} <\frac{1}{p}+\min\left\{\frac{1}{p},\frac{1}{n}\right\}.
   		   	\end{equation}
   	
   \end{prop}


    We introduce a complex geometrical optics (CGO) solution $u_0$ defined by \eqref{1eq:cgo} in Lemma \ref{1lem:nor}.


    \begin{lem}\cite{CX21,LME}\label{1lem:nor}
    	Given the space dimensions $n=2,3$, let $q$ satisfy the assumption in Proposition \ref{1pro:nor} with the constant $p$ subject to $p>\frac{1}{n-1}$ and $\frac{n}{p}<\frac{2}{n+1}+1$. Let
     \begin{equation}\label{1eq:cgo}
     u_{0}(\mathbf x)=(1+\psi (\mathbf x))e^{\rho\cdot \mathbf x},\ \mathbf x\in \mathbb R^n
     \end{equation}
      where
     \begin{equation}\label{1eq:eta}
     \rho=-\tau(\mathbf{d}+\mathrm{i} \mathbf{d}^{\perp}),
     \end{equation}
     with $\mathbf{d},\ \mathbf{d}^{\perp} \in \mathbb{S}^{n-1}$  satisfying $\mathbf{d}\perp \mathbf{d}^{\perp}$,
     and  $\tau \in \mathbb{R}_+$. If $\tau $ is sufficient large, there exits a solution $u_0(\mathbf x)$ with the form (\ref{1eq:cgo}) satisfying
     \begin{equation}\label{1eq:u_0}
     \Delta u_0 +k^2qu_0=0\quad \mbox{in} \quad \mathbb R^n,
     \end{equation}
     and $\psi(\mathbf x)$ fulfills that
      \begin{equation}\label{1nor:psi}
      \vert \vert \psi(\mathbf x)\vert \vert_{H^{1,p}}=\mathcal O\left(\tau^{n(\frac{1}{\tilde p}-\frac{1}{p})-2}\right ),
     \end{equation}
     where $\widetilde p $ satisfies \eqref{eq:p til cond}.

    \end{lem}

     \begin{prop}\cite[Lemma 4.4]{DLW}\label{1prop:gamma}
     	For any given  $\alpha>0$ and $0<\epsilon<e$ , we have the following estimates
     	\begin{subequations}
     	\begin{align}
    & 	\left\vert \int_{\epsilon}^{\infty} r^{\alpha}e^{-\mu r}\mathrm dr \right\vert \leq \frac{2}{\Re\mu}e^{\frac{\epsilon}{2}\Re{\mu}} ,\label{1eq:Ir} \\
& \int_{0}^{\epsilon}r^{\alpha}e^{-\mu r}\mathrm dr =
     	\frac{\Gamma(\alpha+1)}{\mu^{\alpha+1}} +\mathcal O\left( \frac{2}{\Re{\mu}}e^{-\frac{\epsilon}{2}\Re{\mu}}\right),\label{1eq:gamma}
     	     	\end{align}
     	\end{subequations}
     as $\Re(\mu)\rightarrow \infty$, 	where $\Gamma (s)$ stands for the Gamma function.
     	\end{prop}

%

    \begin{lem}\cite{costabel88}\label{2lem:green}
    	Let $\Omega \subset \mathbb R^n$ be a bounded Lipschitz domain. For any $f, g\in H^{1,{\Delta}}:=\{f\in H^{1}(\Omega)~|~\Delta f\in L^{2}(\Omega)\}$, then the following Green formula holds
    	\begin{equation}\label{2eq:2green}
    	\int_{\Omega}(g\Delta f-f\Delta g){\rm d}\mathbf x=\int_{\partial \Omega}(g\partial _{\nu}f-f\partial_\nu g) {\rm d}\sigma,
    	\end{equation}
    	where $\partial_\nu f$ is the exterior normal derivative of $f$ to $\partial \Omega$.
    \end{lem}

    \section{Vanishing of transmission eigenfunctions near a convex  planar  corner}\label{sec:2D}
    In this section, we consider the vanishing property of conductive transmission eigenfunctions to \eqref{0eq:tr} near corners  in $\mathbb R^2$. Firstly, let us introduce some notations for the subsequent use. Let $(r,\theta)$ be the polar coordinates in $\mathbb{R}^{2}$; that is $\mathbf{x}=(x_{1},x_{2})=(r\cos \theta, r\sin \theta)\in \mathbb{R}^{2}$. For $\mathbf{x}\in \mathbb {R}^{2}$, $B_{h}(\mathbf{x})$ denotes an open ball of radius $h \in  \mathbb {R}_{+}$ and centered at $ \mathbf{x} $. For simplicity, we denote $B_{h}\ :=B_{h}(\mathbf{0})$. Consider an open sector in $\mathbb{R}^{2}$ with the boundary $\Gamma^{\pm}$ as follows,
    \begin{equation}\label{1eq:sec}
    \mathcal{K}=\{ \mathbf{x} \in \mathbb{R}^{2}~|~\theta_{m}<\arg(x_{1}+\mathrm ix_{2})<\theta_{M} \},
    \end{equation}
    where $-\pi <\theta_{m}<\theta_{M}<\pi, \mathrm{i}:=\sqrt{-1}$ and the two boundaries $\Gamma^{\pm}$ of $\mathcal K$
     correspond to $(r,\theta_{m})$ and $(r,\theta_{M})$ with $r>0$,  respectively . Set
    \begin{equation}\label{1eq:not}
    S_{h}=\mathcal{K}\cap B_{h},\  \Gamma_{h}^{\pm}=\Gamma^{\pm}\cap B_{h},\ \Lambda_{h}=\mathcal{K}\cap \partial B_{h}.
    \end{equation}

    Let the Herglotz wave function be defined by
    	\begin{align}
    		u(\mathbf x)=\int_{\mathbb S^{1}} e^{\mathrm{i}k\xi\cdot \mathbf{x}}g (\xi)\mathrm d\xi,\ \xi \in \mathbb S^{n-1},\ \mathbf{x} \in \mathbb{R}^n,\ g \in L^{2}(\mathbb S^{n-1}),   \quad n= 2 \mbox{  or }3,
    	\end{align}
    	which  is an entire  solution of
    	$$
    	(\Delta +k^2)u(\mathbf x)=0\quad \mbox{ in } \quad \mathbb R^n,\quad n=2 \mbox{  or }3.
    	$$
  By \cite[Theorem 2 and Remark 2]{Wec},  we  know  that  the set of  the Herglotz wave function is dense with respect to $H^1$ norm  in the set  of the solution to
  $$
 (\Delta +k^2 ) v(\mathbf x)=0 \quad \mbox{ in } \quad D, \quad D\subset \mathbb R^n,\quad n=2 \mbox{  or }3,
  $$ 	
  where $D$ is a bounded  Lipschitz  domain with  a connected  complement.

  Consider the transmission eigenvalue problem \eqref{0eq:tr} defined  in a bounded Lipschitz domain $\Omega$ with a connected complement. Since $\Delta$ is invariant under rigid motions, without loss of generality, we always assume that $ \mathbf 0 \in \partial  \Omega$ throughout of the rest of this paper. In Theorem \ref{thm:2D}, we establish the vanishing property of the transmission eigenfunctions  near  a convex planar corner under $H^1$ regularity with certain Herglotz wave approximation assumptions in the underlying corner. We postpone the proof of Theorem \ref{thm:2D} in the subsection \ref{subsec:3.1}. Compared with the assumptions in \cite[Theorem 2.1]{DCL}, we remove the technical condition $qw\in C^{\alpha}(\overline{S}_h )$, which is critical for the analysis in \cite{DCL}. 


    \begin{thm}\label{thm:2D}
     	Consider a pair of transmission  eigenfunctions $v\in H^{1}(\Omega)$ and $w\in H^{1}(\Omega)$   to (\ref{0eq:tr}) associated with $k\in \mathbb R_+$, where $\Omega$ is a bounded  Lipschitz  domain with  a connected  complement. Suppose that $\mathbf 0\in \Gamma \subset \partial \Omega$ such that $\Omega \cap B_h=\mathcal K \cap B_h=S_h$, where the sector $\mathcal K$ is  defined by (\ref{1eq:sec}) and $h\in \mathbb R_+$ is sufficiently small such that $q\in H^2(\overline{S}_h)$ and $\eta \in C^{\alpha}(\overline{\Gamma_h^\pm} )$, where $\alpha \in (0,1)$.  If the following conditions are fulfilled:
     	\begin{itemize}
     		\item[(a)] for any given positive constants $\beta$ and $\gamma$ satisfying
     		 \begin{equation}\label{eq:assump1}
       	\gamma <\alpha \beta,
       \end{equation}
        the transmission eigenfunction $v$ can be approximated in $H^{1}(S_h)$ by the Herglotz wave functions
     		$$
     		v_j=\int_{\mathbb S^{1}} e^{\mathrm{i}k\xi\cdot \mathbf{x}}g_j (\xi)\mathrm d\xi,j=1,2,\cdots,
     		$$
     		with the kernels $g_j$ satisfying the approximation property
     		\begin{equation}\label{1eq:herg}
    	\|v-v_{j}\|_{H^{1}}\leq j^{-\beta},\quad \|g_{j}\|_{L^{2}(\mathbb S^{1})}\leq j^{\gamma };
    	\end{equation}
    	
    	 \item[(b)] $\eta$ does not vanish at $\mathbf 0$, where $\mathbf 0$ is the vertex of $S_h$;

       \item[(c)] the open angle of $S_h$ satisfies

       \begin{equation}\notag
        -\pi <\theta_m<\theta_M<\pi\ and \ 0<\theta_M-\theta_m <\pi;
       \end{equation}
     	  \end{itemize}
     	then one has
     	\begin{equation}\label{1eq:v0}
  \lim_{\lambda \to +0}\frac{1}{m(B(\mathbf 0,\lambda)\cap \Omega)}\int_{B(\mathbf 0,\lambda)\cap \Omega}\vert v(\mathbf x)\vert \mathrm d\mathbf x=0,
     	\end{equation}
     	 where $m(B(\mathbf 0,\lambda)\cap \Omega)$ is the area of $B(\mathbf 0,\lambda)\cap \Omega$.

    \end{thm}
    
  It is remarked that the Herglotz approximation property in \eqref{1eq:herg} characterises a regularity lower than H\"older continuity (cf. \cite{LT}). In the following theorem, if the stronger H\"older regularity is imposed on the transmission eigenfunction $v$ near the corner is satisfied, we can prove that $v$ vanishes near the corner point. The proof of Theorem \ref{2D:delta} is a slight modification of the corresponding proof of Theorem \ref{thm:2D}. We only give a  sketched proof of Theorem \ref{2D:delta} at the end of Subsection \ref{subsec:3.1}.
    \begin{thm}\label{2D:delta}
       Consider a pair of transmission  eigenfunctions $v\in H^{1}(\Omega)$ and $w\in H^{1}(\Omega)$   to (\ref{0eq:tr}) associated with $k\in \mathbb R_+$, where $\Omega$ is a bounded  Lipschitz  domain with  a connected  complement. Suppose that $\mathbf 0\in\Gamma \subset \partial \Omega$ such that $\Omega \cap B_h=\mathcal K \cap B_h=S_h$, where the sector $\mathcal K$ is  defined by (\ref{1eq:sec}) and $h\in \mathbb R_+$. If the following conditions are fulfilled:
       	\begin{itemize}
       	\item[(a)] $q\in H^2(\overline{S}_h)$, $v\in C^\alpha(\overline{S}_h) $ and $\eta \in C^{\alpha}(\overline{\Gamma_h^\pm} )$, where $0<\alpha<1$;
       		\item[(b)]the function $\eta$ does not vanish at the vertex $\mathbf 0$, where $\mathbf 0$ is the vertex of $S_h$, i.e.,
       	\begin{equation}\label{2eq:q0}
       	\eta(\mathbf 0)\not=0;
       	\end{equation}
       	\item[(c)]the open angels of $S_h$ satisfies
       	$$-\pi <\theta_m<\theta_M<\pi,\ and\ 0<\theta_M-\theta_m<\pi;$$
       	\end{itemize}
       	
       	then one has
       	\begin{equation}\label{1eq:delv0}
       	 v(\mathbf 0) =0.
       	\end{equation}

       \end{thm}

Recall that $\Omega$ is a bounded Lipschitz domain and $\Gamma$ is an open subset of $\partial \Omega$.  Consider the  classical  transmission eigenvalue problem:
\begin{equation}\label{3eq:treta0}
     \begin{cases}
      &\Delta w+k^{2}qw=0  \quad \hspace*{0.55cm} \mbox {$\mathrm {in}\  \Omega$},\\
      &\Delta v+k^{2}v=0   \quad \hspace*{0.9cm}\mbox {$\mathrm {in}\ \Omega$},\\
      &w=v,\ \partial_{\nu}w=\partial_{\nu}v  \hspace*{0.5cm}\mbox{$\mathrm {on}\   \Gamma$},
     \end{cases}
     \end{equation}
   which can be formulated from \eqref{0eq:tr} by setting $\eta\equiv 0$.  When $\Gamma=\partial  \Omega$, \eqref{3eq:treta0} is referred to be {\it interior transmission eigenvalue problem}, which has a colorful history in invere scattering theory (cf.  \cite{CCH,CH2013,Liu22}   and references therein). It was revealed that in \cite[Theorem 1.2]{B2018} the transmission eigenfunction $v$ and $w$ to \eqref{3eq:treta0}   must vanish near a planar corner of $\Gamma$ if $v$ or $ w$ is $H^2$-smooth near the underlying  corner and $q$ is H\"older continuous at the corner point.  In the following Corollary \ref{cor1}, we shall establish the vanishing characterization of  transmission eigenfunctions to \eqref{3eq:treta0}  near a convex planar corner under two regularity criterions on the underlying transmission eigenfunctions near the corner.  We should emphasize that we  remove the $H^2$-smooth near the corner assumption on $v$ and $w$ as stated in \cite[Theorem 1.2]{B2018}, where we only require that $v$ is H\"older continuous at the corner point or holds  a certain regularity condition in terms of Herglotz wave approximations (which is weaker than H\"older continuity as remarked earlier).  The proof of Corollary \ref{cor1} is postponed to Subsection \ref{subsec:32}.

    \begin{cor}\label{cor1}
    		Consider a pair of transmission  eigenfunctions $v\in H^{1}(\Omega)$ and $w\in H^{1}(\Omega)$    to \eqref{3eq:treta0}  associated with $k\in \mathbb R_+$, where $\Omega$ is a bounded  Lipschitz  domain with  a connected  complement. Suppose that $\mathbf 0\in \Gamma \subset  \partial \Omega$ such that $\Omega \cap B_h=\mathcal K \cap B_h=S_h$, where the sector $\mathcal K$ is  defined by (\ref{1eq:sec}) and $h\in \mathbb R_+$ is sufficient small such that $q\in H^2(\overline S_h)$ and $q(\mathbf 0)\neq 1$. The following two statements are valid.
    		\begin{itemize}
    			\item[(a)] For any given positive constants $\beta$ and $\gamma$ satisfying
    			$\gamma <\beta$,
    			if the transmission eigenfunction $v$ and Herglotz wave functions $v_j$ with the kernel $g_j$ satisfying the approximation property \eqref{1eq:herg}, then  we have the vanishing property of $v$ near $S_h$ in the sense of \eqref{1eq:v0}.
    			
    			\item[(b)] If $v\in C^\alpha (\overline{S_h})$ with $\alpha\in  (0,1)$, then it holds that $v(\mathbf 0)=0$.
    		\end{itemize}
    		

    \end{cor}

        \subsection{Proof of Theorem  \ref{thm:2D} }\label{subsec:3.1}
     Given a convex sector $\mathcal K$ defined by \eqref{1eq:sec} and a positive constant $\zeta$, we define $\mathcal{K}_{\zeta}$ as the open set of $\mathbb{S}^{1}$ which is composed of all directions $\mathbf d\in \mathbb{S}^{1}$ satisfying that
    \begin{equation}\label{2eq:zeta}
    \mathbf d\cdot \hat{\mathbf{x}}>\zeta>0, \quad for\ all\ \hat{\mathbf{x}} \in\mathcal{K} \cap\mathbb S^1 .
    \end{equation}
    Throughout the present section, we always assume that the unit vector $ \mathbf d$ in the form of the CGO solution $u_0$ given by \eqref{1eq:cgo} fulfills \eqref{2eq:zeta}.

    \begin{prop}\label{2prop:tau2}
Let $S_{h}$ and $\rho$ be defined in (\ref{1eq:not}) and (\ref{1eq:eta}), respectively, where $ \mathbf d$ satisfies   \eqref{2eq:zeta}. Then we have
    	\begin{equation}\label{1eq:tau2}
    	\left\vert \int_{\Gamma_h^{\pm}}e^{\rho \cdot \mathbf x}\mathrm d \mathbf x\right \vert \geq  \frac{C_{S_h }}{\tau}-\mathcal O\left(\frac{1}{\tau}e^{-\frac{1}{2}\zeta h \tau}\right),
    	\end{equation}
    	for sufficiently large $\tau$, where  $C_{S_h}$ is a positive number only depending on the opening angle $\theta_M-\theta_m$ of $\mathcal K$ and $\zeta$.
    	\end{prop}

    \begin{proof}
    	Using polar coordinates transformation and Proposition \ref{1prop:gamma}, we have
    	\begin{equation}\notag
    	\begin{aligned}
    		\int_{\Gamma_h^{\pm}}e^{\rho\cdot \mathbf x}\mathrm d \sigma
    		&=\frac{\Gamma(1)}{\tau}\frac{1}{(\mathbf d+\mathrm i\mathbf d^{\perp})\cdot \mathbf{\hat x_1}}-I_{R_1}+\frac{\Gamma(1)}{\tau}\frac{1}{(\mathbf d+\mathrm i\mathbf d^{\perp})\cdot \mathbf{\hat x_2}}-I_{R_2},
    	\end{aligned}
    	\end{equation}
    	where $\mathbf {\hat x_1}$ and $\mathbf {\hat x_2}$ are unit  vector of $\mathbf x$ on $\Gamma^-$ and $\Gamma ^+$, and
    	\begin{equation}\notag
    	I_{R_1}=\int_{\Gamma^-\setminus \Gamma^-_h}e^{-\tau (\mathbf d+\mathrm i\mathbf d^\perp)}\mathrm d\sigma,\quad I_{R_2}=\int_{\Gamma^+\setminus \Gamma^+_h}e^{-\tau (\mathbf d+\mathrm i\mathbf d^\perp)}\mathrm d\sigma.
    	\end{equation}
    	Hence, with the help of Proposition \ref{1prop:gamma}, for sufficiently large $\tau$, we have the following integral inequality
    	\begin{equation}\label{1eq:etatau2}
    	\begin{aligned}
    	  \left\vert  \int_{\Gamma_h^{\pm}}e^{\rho\cdot \mathbf x}\mathrm d \sigma \right \vert&\geq\frac{1}{\tau}\left \vert \frac{1}{(\mathbf d+\mathrm i\mathbf d^{\perp})\cdot \mathbf{\hat x_1}} + \frac{1}{(\mathbf d+\mathrm i\mathbf d^{\perp})\cdot \mathbf{\hat x_2}}  \right \vert-\vert I_{R_1}\vert -\vert I_{R_2} \vert\\
    	  &\geq \frac{C_h}{\tau}-O\left(\frac{1}{\tau}e^{-\frac{1}{2}\zeta h\tau}\right)
    	      		\end{aligned}
    	\end{equation}
    	by using \eqref{2eq:zeta}.    
    \end{proof}

The following proposition can be directly derived by using \eqref{2eq:zeta} and Proposition \ref{1prop:gamma}.

\begin{prop}\label{1prop:enorm}
        For any given $t>0$, we let $S_h$ and $\Gamma_h^\pm$ be defined by  \eqref{1eq:not}. Then one has
    	\begin{equation}\label{1eq:enorm}
    	\vert\vert e^{\rho  \cdot \mathbf {x}} \vert \vert_{L^t(S_h)}\leq C \left(\frac{1}{\tau^2}+\frac{1}{\tau}e^{-\frac{t}{2}\zeta h\tau}\right)^{\frac{1}{t}},
    	\end{equation}
    	\begin{equation}\label{1eq:enormGamma}
    	\vert\vert e^{\rho \cdot \mathbf {x}} \vert \vert_{L^t(\Gamma_h^{\pm})}\leq C \left(\frac{1}{\tau}+\frac{1}{\tau}e^{-\frac{t}{2}\zeta h\tau}\right)^{\frac{1}{t}}.
    	\end{equation}
    	as $\tau\rightarrow \infty$,  where $\vert\vert e^{\rho \cdot \mathbf {x}} \vert \vert_{L^t(S_h)}=\left(\int_{S_h} |e^{\rho \cdot \mathbf {x}} |^t {\mathrm d}x\right)^{1/t}$,  $\rho$ is defined in (\ref{1eq:eta}) and $C$ is a positive constant only depending on $t,\zeta$ .
    	
    	\end{prop}

  \begin{lem}\label{lem:31}
  	Under the same setup of Theorem \ref{thm:2D}, let the CGO solution $u_0$ be defined by \eqref{1eq:cgo}. Denote $u=w-v$, where $(v,w)$ is a pair of transmission eigenfunctions of \eqref{0eq:tr} associated with $k$.  Then it holds that
  	\begin{equation}\label{eq:pde sys 2}
  	\begin{cases}
  		&\Delta u_0 +k^2qu_0=0\hspace*{2.1cm}\mbox{in}\quad  S_h,\\
  		&\Delta u +k^2qu=k^2(1-q)v \hspace*{0.9cm} \mbox{in}\quad  S_h,\\
  		&u=0,\quad \partial_\nu u=\eta v  \hspace*{1.9cm} \mbox{on}\quad \Gamma_h^\pm,
  	\end{cases}
  	\end{equation}
  	and \begin{equation}\label{1eq:psi}
    	\vert \vert \psi(\mathbf x) \vert \vert_{H^{1,8}}=\mathcal O(\tau^{-\frac{2}{3}}),
    	\end{equation}
        where $\psi$ and $\tau$ are defined in \eqref{1eq:cgo}.
  \end{lem}
  \begin{proof}
  	Since $q \in H^2(\overline{ S}_h)$, let $\widetilde q $ be the Sobolev extension of $q$, one has $\widetilde q \in H^2$. Hence we have $\widetilde q\in H^{1,1+\epsilon_0}$ where $\epsilon_{0}=\frac{1}{2}$. Therefore $\widetilde q$ satisfies the assumption in  Proposition \ref{1pro:nor}. Let  $\tilde{p}=1+\frac{10}{19}\epsilon_{0}$, according to \eqref{eq:p til cond},   one has $p=8$. Since $\widetilde q$ and $p$ fulfill the assumption of   Lemma \ref{1lem:nor},  there exits a CGO solution with the form \eqref{1eq:cgo} satisfies
  	\begin{align}\label{eq:cgo 2}
  	 \Delta u_0 +k^2\widetilde qu_0=0\quad \mbox{in} \quad \mathbb R^n,\quad \widetilde q~\big|_{S_h} =q.
  	\end{align}
  	By \eqref{1nor:psi}, it yields that \eqref{1eq:psi}. Using \eqref{0eq:tr} and \eqref{eq:cgo 2}, we can obtain \eqref{eq:pde sys 2}.
  \end{proof}


\begin{lem}\label{lem:trace thm}\cite{E.G,LME,GG07,SEM}
	Let $\Omega$ be a Lipschitz bounded and connected subset of $\mathbb R^n,n=2,3$ whose bounded and orientable boundary is denote by $\Gamma$. Let the restriction $\gamma_0(u)=u|_\Gamma$, then the operator $\gamma_0$ is linear and continuous from $H^{1,p}(\Omega)$ onto $H^{1-\frac{1}{p},p}(\Gamma)$ for $1\leq p<\infty$.
\end{lem}

    \begin{lem}\label{2lem:psi8}
    	Let $\Gamma_h^\pm$ be defined in \eqref{1eq:not}, $e^{\rho \cdot \mathbf x}$ and $\psi$ be given by \eqref{1eq:cgo} and \eqref{1eq:eta}. For sufficiently large $\tau$, it holds that
    \begin{equation}
    \left\vert \int_{\Gamma_h^\pm} e^{\rho \cdot \mathbf x}\psi(\mathbf x)\mathrm d\sigma\right\vert \lesssim \tau^{-\frac{17}{12}}.
    \end{equation}
    Throughout of the rest of this paper,  $\lesssim$ means that we only give the leading asymptotic analysis by neglecting a generic positive constant $C$ with respect to $\tau\rightarrow \infty$, where $C$ is not a function of $\tau$.
    \end{lem}

    \begin{proof}
    Taking $\mathbf y=\tau \mathbf x $, then  using H\"older inequality and Lemma \ref{lem:trace thm}, one has
    	\begin{align}\label{2eq:etagamma}
    	\int_{\Gamma_h^\pm}\vert e^{\rho \cdot \mathbf x}\vert&\vert \psi(\mathbf x)\vert \mathrm d\sigma
    	\lesssim \frac{1}{\tau}\|e^{-\mathbf d\cdot \mathbf y}\|_{L^{\frac{8}{7}}(\Gamma_{\tau h}^\pm)}\left\|\psi\left(\frac{\mathbf y}{\tau}\right)\right\|_{L^8(\Gamma_{\tau h}^\pm)} \\
    	&
    	\lesssim \frac{1}{\tau}\|e^{-\mathbf d\cdot \mathbf y}\|_{L^{\frac{8}{7}}(\Gamma^\pm)}\left\|\psi\left(\frac{\mathbf y}{\tau}\right)\right\|_{H^{1,8}(S_{\tau h})}  \lesssim \frac{1}{\tau}\|e^{-\mathbf d\cdot \mathbf y}\|_{L^{\frac{8}{7}}(\Gamma^\pm)}\left\|\psi\left(\frac{\mathbf y}{\tau}\right)\right\|_{H^{1,8}(\mathcal K)},\notag
    	\end{align}
         for sufficiently large $\tau $. We have $\frac{1}{\tau}<1$, and it holds that
    	\begin{align}\label{2eq:psik}
       \left\|\psi\left(\frac{\mathbf y}{\tau}\right)\right\|_{H^{1,8}(\mathcal K)}
        &\leq\tau^{\frac{1}{4}}\left \|\psi(\mathbf x)\right \|_{H^{1,8}(\mathcal K)}=\mathcal O(\tau^{-\frac{5}{12}}),\quad \mbox{as  $\tau \rightarrow \infty$. }
        \end{align}
    	Furthermore,
    	\begin{equation}\label{2eq:edy}
    	\|e^{-\mathbf d\cdot \mathbf y}\|_{L^{\frac{8}{7}}(\Gamma ^\pm )}
    	\leq \left(\int_{\Gamma^\pm}e^{-\frac{8}{7}\zeta \vert \mathbf y\vert}\mathrm d\sigma\right)^\frac{7}{8}
    	=2\left( \frac{7}{8}\frac{1}{\zeta}\right)^{\frac{7}{8}},
    	\end{equation}
    	where $\zeta$ is defined in \eqref{2eq:zeta}. Hence, $\|e^{-\mathbf d\cdot \mathbf y}\|_{L^{\frac{8}{7}}(\Gamma ^\pm )}$ is a positive constant which only depends on $\zeta$.
    	Combining  \eqref{2eq:psik} and \eqref{2eq:edy} with \eqref{2eq:etagamma}, we can prove Lemma \ref{2lem:psi8}.
    \end{proof}

     \begin{lem}\label{2lem:u0est}
     		Let $\Lambda_h,\ S_h$ be defined in \eqref{1eq:not} and $u_0(\mathbf x)$ be given by \eqref{1eq:cgo}. Then $u_0(\mathbf x)\in H^1(S_h)$ and it holds that
     	   \begin{subequations}
     	   	\begin{align}
     	   		\|u_0(\mathbf x)\|_{L^2(\Lambda_h)}&\lesssim \left(1+\tau^{-\frac{2}{3}}\right)e^{-\zeta h\tau},\label{2eq:u0lambda} \\
     	   		\label{2eq:nablau0}
     	   		\|\nabla u_0(\mathbf x)\|_{L^2(\Lambda_h)}&\lesssim (1+\tau) \left(1+\tau^{-\frac{2}{3}}\right)e^{-\zeta h\tau},\\
     	   		\label{2eq:u0alpha}
     	   		\int_{S_h}\vert \mathbf x\vert^{\alpha}\vert u_0(\mathbf x)\vert\rm d\mathbf x
     	   		&\lesssim
     	   		\tau^{-(\alpha+\frac{29}{12})}+\left(\frac{1}{\tau^{\alpha+2}}+\frac{1}{\tau}e^{-\frac{1}{2}\zeta h\tau}\right)
     	   	\end{align}
     	   \end{subequations}
     		as $\tau \rightarrow \infty$, where $\zeta$  is defined  in \eqref{2eq:zeta} and $\alpha \in (0,1)$.
%
%
%
     \end{lem}
     \begin{proof} Using  polar coordinates transformation, \eqref{1eq:Ir} and  \eqref{2eq:zeta}, we can obtain that
     \begin{equation}\label{2eq:eeta}
     	\|e^{\rho \cdot \mathbf x}\|_{ L^{t}(\Lambda_{h})} \lesssim e^{-\zeta h\tau}.
     \end{equation}
     where $\rho$  is defined  in \eqref{1eq:eta} and $t$ is a positive constant.

     According to \eqref{1eq:psi} and  Lemma \ref{lem:trace thm}, for sufficient  large $\tau$, one can show that
     \begin{align}\label{1eq:psi est}
     	\|\psi(\mathbf x)\|_{L^{4}(\Lambda_{h})}
     	\leq C \| \psi(\mathbf x)\|_{H^{1,8}(S_h)}=O(\tau^{-\frac{2}{3}}),
     \end{align}
     where $C$ is a positive constant, which is not a function of $\tau$.

     By virtue of \eqref{1eq:psi est} and  H\"older inequality, it can be directly verified that
     	\begin{equation}\label{2eq:I42}
     	\begin{aligned}
     	\|u_{0}\|_{ L^{2}(\Lambda_h)}
     	&\leq\|e^{\rho \cdot \mathbf x}\|_{ L^{2}(\Lambda_{h})}+\|e^{\rho \cdot \mathbf x}\|_{ L^{4}(\Lambda_{h})}\|\psi(\mathbf x)\|_{ L^{4}(\Lambda_{h})}\\
     	&\lesssim \left(1+\tau^{-\frac{2}{3}}\right)e^{-\zeta h\tau},\quad \mbox{as  $\tau \rightarrow \infty$. }
     	\end{aligned}
     	\end{equation}
    Similarly, using Cauchy-Schwarz inequality, \eqref{2eq:u0lambda} and Proposition \ref{1prop:gamma}, we have
     	\begin{equation}\label{2eq:I43}
     	\begin{aligned}
     	\|\nabla u_{0}(\mathbf x)\|_{L^{2}(\Lambda_h)}
     	&\leq \sqrt{2}\tau \|u_{0}\|_{ L^{2}(\Lambda_h)}+\|e^{\rho \cdot \mathbf x}\|_{ L^{4}(\Lambda_h)} \|\nabla \psi(\mathbf x)\|_{ L^{4}(\Lambda_h)}\\
     	&\lesssim(1+\tau)(1+\tau^{-\frac{2}{3}})e^{-\zeta h\tau},\quad \mbox{as  $\tau \rightarrow \infty$. }
     	\end{aligned}
     	\end{equation}
     	Moreover, by using  Cauchy-Schwarz inequality, we know that
     	\begin{equation}\label{2eq:xalpha}
     	\begin{aligned}
     	\int_{S_{h}}\vert \mathbf x\vert^{\alpha}\vert u_{0} \vert\mathrm{d}\mathbf x
     	&\leq\int_{S_{h}}\vert \mathbf x\vert^{\alpha}\vert e^{\rho \cdot \mathbf{x}} \vert \mathrm{d}\mathbf x+\int_{S_{h}}\vert \mathbf x\vert^{\alpha}\vert e^{\rho \cdot \mathbf{x}}\vert\vert \psi(\mathbf x)\vert\mathrm{d}\mathbf x.\\
     	\end{aligned}
     	\end{equation}
     	Using polar coordinates transformation  and Proposition \ref{1prop:gamma}, we can deduce that
     	\begin{equation}\label{2eq:I111}
     	\int_{S_{h}}\vert \mathbf x\vert^{\alpha}\vert e^{\rho \cdot \mathbf{x}} \vert \mathrm{d}\mathbf x
     	\lesssim \left(\frac{1}{\tau^{\alpha+2}}+\frac{1}{\tau}\right)e^{-\frac{1}{2}\zeta h\tau },\quad \mbox{as  $\tau \rightarrow \infty$. }
     	\end{equation}
        Next, by letting $\mathbf y=\tau \mathbf x$ and H\"older inequality, it can be calculated that
        \begin{equation}\label{2eq:u0K}
        \begin{aligned}
        \int_{S_{h}}\vert \mathbf x\vert^{\alpha}\vert e^{\rho \cdot \mathbf{x}}\vert\vert \psi(\mathbf x)\vert\mathrm{d}\mathbf x
        &\leq \frac{1}{\tau^{\alpha+2}}\int_{\mathcal K}\vert \mathbf y\vert^{\alpha}\vert e^{-\mathbf d\cdot \mathbf y}\vert \left\vert \psi\left(\frac{\mathbf y}{\tau}\right)\right\vert\mathrm d\mathbf x\\
        &\leq \frac{1}{\tau^{\alpha+2}}\|\vert \mathbf y\vert^{\alpha}\vert e^{-\mathbf d\cdot \mathbf y}\vert \|_{L^{\frac{8}{7}}(\mathcal K)}\left\|\psi\left(\frac{\mathbf y}{\tau}\right)\right\|_{L^{8}(\mathcal K)}.
        \end{aligned}
        \end{equation}
        With the help of variable substitution and \eqref{1eq:psi}, we can calculate that
        \begin{equation}\label{2eq:psiL8K}
        \left\|\psi\left(\frac{\mathbf y}{\tau}\right)\right\|_{L^8(\mathcal K)}=\tau^{\frac{1}{4}}\|\psi\left(\mathbf x\right)\|_{L^8(\mathcal K)}=\mathcal O(\tau^{-\frac{5}{12}}),\quad \mbox{as  $\tau \rightarrow \infty$. }
        \end{equation}	
        Similar to \eqref{2eq:edy}, by using polar coordinates transformation , we have
        $
        \| \vert \mathbf y\vert^\alpha \vert e^{-\mathbf d\cdot \mathbf y}\vert \|_{L^\frac{8}{7}(\mathcal K)}
        $
        is a positive constant and not a function of $\tau$.
        Therefore, combining  \eqref{2eq:psiL8K}, \eqref{2eq:u0K} and \eqref{2eq:I111} with  \eqref{2eq:xalpha}, we have \eqref{2eq:u0alpha}.

        The proof is complete.
     \end{proof}

     Now we are in a position to prove Theorem \ref{thm:2D}.

    \begin{proof}[{\bf Proof of Theorem \ref{thm:2D}}]
    	By Green's formula \eqref{2eq:2green} and \eqref{eq:pde sys 2},  the following integral equality holds
    	\begin{equation}\label{1eq:green}
    	\begin{aligned}
    	\int_{\Lambda_{h}}(w-v)\partial_{\nu}u_{0}-u_{0}\partial_{\nu}(w-v)\mathrm{d}\sigma-\int_{\Gamma_h^{\pm}}\eta u_0v\mathrm d\sigma
    	&=k^{2}\int_{S_{h}}(q-1)vu_{0}\mathrm{d}\mathbf x.
    	\end{aligned}
    	\end{equation}
    	Denote
    	$$
    	f_{j}=(q-1)v_{j}.
    	$$
    Since $q\in H^2(\overline S_h)$, by Sobolev embedding property, one has $q\in C^{\alpha}(\overline S_h)$ where $\alpha \in (0,1]$. Clearly, $v_{j}\in C^{\alpha}(\overline S_h)$, hence $f_{j}\in C^{\alpha}(\overline S_h)$. According $v_j\in C^{\alpha},\ \eta\in C^{\alpha}$, we have the expansion
    	\begin{equation}\label{1eq:fj}
    	\begin{aligned}
    	f_{j}&=f_{j}(\mathbf{0})+\delta f_{j},\ \vert \delta f_{j}\vert\leq \| f_{j}\|_{C^{\alpha}( {\overline S}_h)}\vert \mathbf{x} \vert^{\alpha},\\
    	v_j&=v_j(\mathbf 0)+\delta v_j,\ \vert\delta v_j\vert \leq\|v_j\|_{C^{\alpha}(\overline S_h)}\vert\mathbf x\vert^{\alpha},\\
    	\eta&=\eta(\mathbf 0)+\delta\eta,\  \vert\delta\eta\vert\leq\|\eta\|_{C^{\alpha}(\overline{ \Gamma_h^\pm }  )}\vert\mathbf x\vert^{\alpha}.
    	\end{aligned}
    	\end{equation}
    	
  By virtue of \eqref{1eq:fj} and  \eqref{1eq:cgo}, it yields that
    	\begin{equation}\label{2eq:fj}
    	\begin{aligned}
    	&k^2\int_{S_{h}}(q-1)vu_0\mathrm d\mathbf x =-\sum_{m=1}^3 I_{m},\quad \int_{\Gamma_h^{\pm}}\eta u_0v\mathrm d\sigma=I-\sum_{m=4}^9I_m,
    	\end{aligned}
    	\end{equation}
    	where
    	\begin{align*}
    	I_{1}&=-k^{2}\int_{S_{h}}(q-1)(v-v_{j})u_{0}\mathrm d\mathbf x
    	,\quad
    	I_{2}=-\int_{S_{h}}\delta f_{j}u_{0}\mathrm d\mathbf x,\quad
    	\\
    	I_{3}&=-f_{j}(\mathbf{0})\int_{S_{h}}u_0\mathrm{d}\mathbf x,\quad
    	I_4=-\eta(\mathbf 0)\int_{\Gamma_h^{\pm}}(v-v_j)u_0\mathrm{d}\sigma,\\
    	I_5&=-\int_{\Gamma_h^{\pm}}\delta \eta(v-v_j)u_0\mathrm{d}\sigma,\quad
    	I_6=-\eta(\mathbf 0)v_j(\mathbf 0)\int_{\Gamma_h^{\pm}}e^{\rho\cdot \mathbf x}\psi(\mathbf x)\mathrm{d}\sigma,\\
    	I_{7}&=-\eta(\mathbf 0)\int_{\Gamma_h^{\pm}}\delta v_ju_0\mathrm d\sigma,\quad
    	I_8=-v_j(\mathbf 0) \int_{\Gamma_h^{\pm}}\delta \eta u_0\mathrm d\sigma,\\
    	I_9&=-\int_{\Gamma_h^{\pm}}\delta \eta\delta v_j u_0\mathrm d\sigma,\quad
    	I=\eta(\mathbf 0)v_j(\mathbf 0)\int_{\Gamma_h^{\pm}}e^{\rho\cdot\mathbf x}\mathrm d\sigma.
    	\end{align*}
    Substituting  \eqref{2eq:fj} into \eqref{1eq:green},  we have the following integral identity
    	\begin{align}\notag
    		I=\sum_{m=1}^9I_m+J_1+J_2,
    	\end{align}
    	where
    	\begin{equation}\label{2eq:chaifen}
    	\begin{aligned}
    	J_{1}&=\int_{\Lambda_{h}}(w-v)\partial_{\nu}u_{0}\mathrm d\sigma,\quad
    	J_{2}=-\int_{\Lambda_{h}} u_{0}\partial_{\nu}(w-v)\mathrm d\sigma.
    	\end{aligned}
    	\end{equation}
    	Therefore, it yields that
    	\begin{equation}\label{2eq:I}
    	\vert I\vert\leq \sum_{m=1}^9\vert I_m\vert+\vert J_1\vert+\vert J_2\vert.
    	\end{equation}
    	
    	In the following, we give detailed asymptotic estimates of $I_m,m=1,\cdots,9$ and $J_j,\ j=1,2$  as $\tau\to \infty$, separately.
    	With the help of Proposition \ref{1prop:enorm}, H\"older inequality and \eqref{1eq:psi}, it arrives at
    	\begin{equation}\label{2eq:I1}
    	\begin{aligned}
    	\vert I_{1} \vert 
    	&\lesssim   \|v-v_j\|_{L^2(S_h)}\left(\|e^{\rho\cdot \mathbf x}\|_{L^2(S_h)}+\|	e^{\rho\cdot \mathbf x}\psi(\mathbf x)\|_{L^2(S_h)}\right)\\
    	&\lesssim \|v-v_j\|_{L^2(S_h)}\left(\|e^{\rho\cdot \mathbf x}\|_{L^2(S_h)}+\|	e^{\rho\cdot \mathbf x}\|_{L^4(S_h)}\|\psi(\mathbf x)\|_{L^4(S_h)}\right)\\
    	&\lesssim j^{-\beta}\left[\left(\frac{1}{\tau^2}+\frac{1}{\tau}e^{-\zeta h\tau}\right)^{\frac{1}{2}}+\left(\frac{1}{\tau^2}+\frac{1}{\tau}e^{-2\zeta h\tau}\right)^{\frac{1}{4}}\tau^{-\frac{2}{3}} \right]
    	\end{aligned}
    	\end{equation}
    	as $\tau \rightarrow\infty$.

    	By virtue of (\ref{1eq:fj}), it yields that
    	\begin{equation}\label{2eq:I21}
    	\vert I_{2} \vert \leq \|f_{j} \|_{C^{\alpha}( \overline S_h)}\int_{S_{h}}\vert \mathbf x\vert^{\alpha}\vert u_{0} \vert\mathrm d\mathbf x,
    	\end{equation}
    where
    	\begin{equation}\label{2eq:fj1}
    	\begin{aligned}
    	\|f_{j}\|_{C^{\alpha}
    	(S_h)}&
    	\leq k^2\left(\|q\|_{C^{\alpha}(\overline S_h)}\sup_{S_h}\vert v_{j} \vert+\|v_{j}\|_{C^{\alpha}( \overline S_h)}\sup_{S_h}\vert q-1\vert\right).
    	\end{aligned}
    	\end{equation}
    	Using the property of compact embedding of H\"older spaces, we can derive that
    	\begin{equation}\label{2eq:calpha}
    	\| v_{j}\|_{C^{\alpha}}\leq {\mathrm {diam} }\left(S_{h}\right)^{1-\alpha}\|v_{j}\|_{C^{1}(S_h)},
    	\end{equation}
    	where diam$\left(S_{h}\right)$ is the diameter of $S_{h}$. By direct computations, we obtain
    	\begin{equation}\label{2eq:vjc1}
    	\|v_{j}\|_{C^{1}}\leq \sqrt{2\pi}(1+k)\|g\|_{L^{2}(\mathbb S^{1})}.
    	\end{equation}
    	Furthermore, by Cauchy-Schwarz inequality,  we also can deduce that
    	\begin{equation}\label{2eq:vj}
    	\vert v_{j}\vert\leq\sqrt{2\pi}\|g\|_{L^{2}(\mathbb S^{1})}.
    	\end{equation}
  Due to \eqref{1eq:herg},     by using the fact that $q\in C^{\alpha}(\overline S_h)$,  substituting  (\ref{2eq:calpha}), (\ref{2eq:vjc1}), and (\ref{2eq:vj}) into (\ref{2eq:fj1}), we have
        \begin{equation}\label{2eq:fjc}
        \|f_{j}\|_{C^{\alpha}(\overline S_h)}\lesssim j^{\gamma },\quad \|v_j\|_{C^\alpha(\overline S_h)}\lesssim j^{\gamma},
        \end{equation}
        where $\gamma$ is a given positive constant  defined in \eqref{1eq:herg}.  Substituting \eqref{2eq:u0alpha} and \eqref{2eq:fjc} into \eqref{2eq:I21}, we can deduce that
        \begin{equation}\label{2eq:I2}
        \vert I_{2} \vert \lesssim j^{\gamma }\left[\tau^{-(\alpha+\frac{29}{12})}+\left(\frac{1}{\tau^{\alpha+2}}+\frac{1}{\tau}e^{-\frac{1}{2}\zeta h\tau}\right)\right]
        \end{equation}
        as $\tau \rightarrow \infty$.

        Using Cauchy-Schwarz inequality, it can be easily calculated that
       \begin{equation}\label{2eq:I31}
       \vert I_{3} \vert
       \lesssim \int_{S_h}\vert e^{\rho \cdot \mathbf x}\vert\mathrm d\mathbf x+\int_{S_h}\vert e^{\rho \cdot \mathbf x}\psi(\mathbf x)\vert \mathrm d\mathbf x\lesssim \int_{S_h}\vert e^{\rho \cdot \mathbf x}\vert \mathrm d\mathbf x+\int_{\mathcal K}\vert e^{\rho \cdot \mathbf x}\psi(\mathbf x)\vert\mathrm d\mathbf x.
       \end{equation}
 By integral substitution and using \eqref{2eq:psiL8K},  we obtain that
        \begin{equation}\label{2eq:I33}
        \begin{aligned}
        \int_{\mathcal K}\vert e^{\rho \cdot \mathbf x}\psi(\mathbf x)\vert\mathrm d\mathbf x&=\frac{1}{\tau^2}\int_{\mathcal K}e^{-\mathbf d\cdot \mathbf y}\left \vert \psi\left(\frac{\mathbf y}{\tau}\right)\right \vert\mathrm d\mathbf y
        \leq \frac{1}{\tau^2} \|e^{-\mathbf d\cdot \mathbf y}\|_{L^{\frac{8}{7}}(\mathcal K)}\|\psi\left(\frac{\mathbf y}{\tau}\right)\|_{L^{8}(\mathcal K)}\\
        &\lesssim \tau^{-\frac{29}{12}},\quad \mbox{as  $\tau \rightarrow \infty$. }
        \end{aligned}
        	\end{equation}
      With the help of Proposition \ref{2prop:tau2}, substituting  (\ref{2eq:I33}) into (\ref{2eq:I31}), we can derive that
        \begin{equation}\label{2eq:I3}
        \vert I_{3} \vert \lesssim \tau^{-\frac{29}{12}}+\left(\frac{1}{\tau^2}+\frac{1}{\tau}e^{-\frac{1}{2}\zeta h\tau}\right), \mbox{as  $\tau \rightarrow \infty$. }
        \end{equation}

        Using Cauchy-Schwarz inequality, the trace theorem and H\"older inequality, we have
        \begin{equation}\label{2eq:I4}
        \begin{aligned}
        \vert I_4\vert
        &\lesssim \|v-v_j\|_{L^2(\Gamma_h^{\pm})}(\|e^{\rho\cdot \mathbf x}\|_{L^2(\Gamma_h^{\pm})}+\|e^{\rho\cdot \mathbf x}\|_{L^4(\Gamma_h^{\pm})}\|\psi(\mathbf x)\|_{L^4(\Gamma_h^{\pm})})\\
        &\lesssim j^{-\beta}\left[\left(\frac{1}{\tau}+\frac{1}{\tau}e^{-\zeta h\tau}\right)^{\frac{1}{2}}+\left(\frac{1}{\tau}+\frac{1}{\tau}e^{-2\zeta h\tau}\right)^{\frac{1}{4}}\tau^{-\frac{2}{3}}\right], 
        \end{aligned}
        \end{equation}
as  $\tau \rightarrow \infty$.  Similarly, by virtue of Cauchy-Schwarz inequality, the trace theorem and H\"older inequality, it can be calculated that
        \begin{equation}
        \begin{aligned}\notag
        \vert I_5\vert
        &\lesssim \|v-v_j\|_{H^1(S_h)}\left(\|e^{\rho\cdot \mathbf x}\vert \mathbf  x\vert^{\alpha}\|_{L^2(\Gamma_h^{\pm})}+\|e^{\rho\cdot \mathbf x}\vert \mathbf  x\vert^{\alpha}\|_{L^4(\Gamma_h^{\pm})}\|\psi(\mathbf x)\|_{L^4(\Gamma_h^{\pm})}\right)\\
        &\lesssim j^{-\beta}\left[ \left(\frac{1}{\tau^{(2\alpha+1)}}+\frac{1}{\tau}e^{-\zeta h\tau}\right)^\frac{1}{2}+\left(\frac{1}{\tau^{(4\alpha+1)}}+\frac{1}{\tau}e^{-2\zeta h\tau} \right)^{\frac{1}{4}}\tau^{-\frac{2}{3}}\right], 
        \end{aligned}
        \end{equation}
        as  $\tau \rightarrow \infty$.  
        By using Lemma \ref{2lem:psi8}, one can show that
        \begin{equation}\label{2eq:I63}
       \begin{aligned}
       I_6&\lesssim \int_{\Gamma_h^\pm}\vert e^{\rho \cdot \mathbf x}\psi(\mathbf x)\vert \mathrm d\mathbf x\lesssim \tau^{-\frac{17}{12}}, \quad \mbox{as  $\tau \rightarrow \infty$. }
       \end{aligned}
       \end{equation}
       Using \eqref{1eq:fj}, \eqref{2eq:fjc}, and Proposition \ref{1prop:gamma}, we have the following inequality
       \begin{align}\label{2eq:I7}
       \vert I_7\vert
       &\lesssim j^{\gamma}\left[\int_{\Gamma_{h}^{\pm}}\vert \mathbf x\vert^{\alpha}\vert e^{\rho \cdot \mathbf x}\vert\mathrm d\sigma+\|\vert \mathbf x\|^{\alpha}\vert e^{\rho\cdot \mathbf x}\vert\|_{L^2(\Gamma_h^{\pm})}\|\psi(\mathbf x)\|_{L^2(\Gamma_h^{\pm})} \right] \nonumber \\
       &\lesssim j^{\gamma}\left[\left(\frac{1}{\tau^{\alpha+1}}+\frac{1}{\tau}e^{-\frac{1}{2}\zeta h\tau}\right)+\left(\frac{1}{\tau^{2\alpha+1}}+\frac{1}{\tau}e^{-\zeta h\tau}\right)^{\frac{1}{2}}\tau^{-\frac{2}{3}}\right], 
       \end{align}
    as  $\tau \rightarrow \infty$.  
       According to \eqref{2eq:I7}, we can derive that
       \begin{subequations}
       \begin{align}\label{2eq:I8}
     &  \vert I_8\vert\lesssim \left(\frac{1}{\tau^{\alpha+1}}+\frac{1}{\tau}e^{-\frac{1}{2}\zeta h\tau}\right)+\left(\frac{1}{\tau^{2\alpha+1}}+\frac{1}{\tau}e^{-\zeta h\tau}\right)^{\frac{1}{2}}\tau^{-\frac{2}{3}},\\
     &  \vert I_9\vert\lesssim
       j^{\gamma}\left[\left(\frac{1}{\tau^{2\alpha+1}}+\frac{1}{\tau}e^{-\frac{1}{2}\zeta h\tau}\right)+\left(\frac{1}{\tau^{4\alpha+1}}+\frac{1}{\tau}e^{-\zeta h\tau}\right)^{\frac{1}{2}}\tau^{-\frac{2}{3}}\right],  \label{2eq:I9}
       \end{align}
\end{subequations}
as  $\tau \rightarrow \infty$.  By the Cauchy-Schwarz inequality and the trace theorem,  we deduce that
        \begin{align}
        \vert J_1\vert&
        \leq C\|u_{0}\|_{H^{1}(\Lambda_{h})} \|w-v\|_{H^{1}(\Lambda_{h})} \label{2eq:J11} \\
        &\lesssim \|u_{0}\|_{H^{1}(\Lambda_{h})} \nonumber
        \end{align}
        as $\tau \rightarrow \infty$, where $C$ is a positive constant arising from the trace theorem.  Hence, by virtue of  \eqref{2eq:u0lambda} and \eqref{2eq:nablau0}, from \eqref{2eq:J11}, it is readily known that
        \begin{equation}\label{2eq:J1}
        \begin{aligned}
        \vert J_1 \vert \lesssim  (1+\tau)(1+\tau^{-\frac{2}{3}})e^{-\zeta h\tau}
        \end{aligned}
        \end{equation}
        as $\tau \rightarrow \infty$, where $\zeta$ is a positive constant given in \eqref{2eq:zeta}.

  Similarly, using Cauchy-Schwarz inequality, the trace theorem and \eqref{2eq:nablau0}, we can obtain that
       \begin{align}
       \vert J_2 \vert &\leq \|\partial _{\nu} u_0\|_{ L^{2}(\Lambda_h)} \|w-v\|_{ L^{2}(\Lambda_h)}
       \leq C \|\partial _{\nu} u_0\|_{ L^{2}(\Lambda_h)} \|w-v\|_{ H^{1}(S_{h})} \label{2eq:J2} \\
       &\lesssim \|\nabla u_{0}\|_{ L^{2}(\Lambda_h)} \lesssim (1+\tau)(1+\tau^{-\frac{2}{3}})e^{-\zeta h\tau}. \nonumber
       \end{align}

       Substituting (\ref{2eq:I1}), (\ref{2eq:I2}), \eqref{2eq:I3}, (\ref{2eq:J1}), and  (\ref{2eq:J2}) into (\ref{2eq:I}), by virtue of (\ref{1eq:tau2}), we derive that
       \begin{align}
         \left(\frac{C_{S_h}}{\tau}-\frac{1}{\tau}e^{-\frac{1}{2}\zeta h\tau}\right)&\vert \eta(\mathbf 0)v_j(\mathbf 0)\vert
         \lesssim
         j^{-\beta}\left[\left(\frac{1}{\tau^2}+\frac{1}{\tau}e^{-\zeta h\tau}\right)^{\frac{1}{2}}+\left(\frac{1}{\tau^2}+\frac{1}{\tau}e^{-2\zeta h\tau}\right)^{\frac{1}{4}}\tau^{-\frac{2}{3}} \right ] \nonumber \\
         & + j^{\gamma }\left[\tau^{-(\alpha+\frac{29}{12})}+\left(\frac{1}{\tau^{\alpha+2}}+\frac{1}{\tau}e^{-\frac{1}{2}\zeta h\tau}\right)\right] \nonumber  \\
         &+j^{-\beta}\left[\left(\frac{1}{\tau}+\frac{1}{\tau}e^{-\zeta h\tau}\right)^{\frac{1}{2}}+\left(\frac{1}{\tau}+\frac{1}{\tau}e^{-2\zeta h\tau}\right)^{\frac{1}{4}}\tau^{-\frac{2}{3}}\right] \label{2D:1} \\
         &+j^{-\beta}\left[ \left(\frac{1}{\tau^{(2\alpha+1)}}+\frac{1}{\tau}e^{-\zeta h\tau}\right)^\frac{1}{2}+\left(\frac{1}{\tau^{(4\alpha+1)}}+\frac{1}{\tau}e^{-2\zeta h\tau} \right)^{\frac{1}{4}}\tau^{-\frac{2}{3}}\right] \nonumber \\
         &+(j^{\gamma}+1)\left[\left(\frac{1}{\tau^{\alpha+1}}+\frac{1}{\tau}e^{-\frac{1}{2}\zeta h\tau}\right)+\left(\frac{1}{\tau^{2\alpha+1}}+\frac{1}{\tau}e^{-\zeta h\tau}\right)^{\frac{1}{2}}\tau^{-\frac{2}{3}}\right] \nonumber \\
         &+j^{\gamma}\left[\left(\frac{1}{\tau^{2\alpha+1}}+\frac{1}{\tau}e^{-\frac{1}{2}\zeta h\tau}\right)+\left(\frac{1}{\tau^{4\alpha+1}}+\frac{1}{\tau}e^{-\zeta h\tau}\right)^{\frac{1}{2}}\tau^{-\frac{2}{3}}\right]  \nonumber \\
         &+(1+\tau)(1+\tau^{-\frac{2}{3}})e^{-\zeta h\tau}+
         \tau^{-\frac{29}{12}}+\left(\frac{1}{\tau^2}+\frac{1}{\tau}e^{-\frac{1}{2}\zeta h\tau}\right)+\tau^{-\frac{17}{12}} \nonumber
       	\end{align}
       as $\tau \rightarrow \infty$, where $C_{S_h}$ is a positive constant given in \eqref{1eq:tau2}.
       Multiplying $\tau$ on both sides of  (\ref{2D:1})  and letting  $\tau=j^s$, where  $s>0$, it can be derived that
       \begin{equation}\label{2D:2}
       \begin{aligned}
       \left(C_h-e^{-\frac{1}{2}\zeta hj^s}\right)\vert \eta(\mathbf 0)v_j(\mathbf 0)\vert&\lesssim j^{-\beta+s}+j^{\gamma-(\alpha+1)s}+j^{-\beta+\frac{1}{2}s}+j^{-\beta+(-\alpha +\frac{1}{2})s}\\
       &+j^{\gamma-\alpha s}+j^{-\frac{13}{24}s}+j^{-\frac{17}{12}s},
       \end{aligned}
       \end{equation}
       as $\tau \rightarrow \infty$. Under the assumption \eqref{eq:assump1}, we can choose $s \in (\gamma/ \alpha , \beta )$.
       Hence in \eqref{2D:2}, let $j\rightarrow \infty$ it is readily to know  that
       $$
       \lim_{j \rightarrow \infty } \left\vert \eta(\mathbf 0)v_j(\mathbf 0)\right\vert=0.$$
       Since $\eta(\mathbf 0)\not = 0$, one has  $\lim_{j \rightarrow \infty }\vert v_j(\mathbf 0)\vert=0 $.  Using \eqref{1eq:herg} and the integral mean value theorem, we can obtain \eqref{1eq:v0}.


The proof is complete.
       \end{proof}


       \begin{proof}[\bf Proof of Theorem \ref{2D:delta}]


     Due to $q\in H^2(\overline S_h)$,  using the Sobolev embedding property, we  know that $q\in C^\alpha(\overline{S_h})$ with $\alpha\in (0,1]$.  Under the assumption $v\in C^\alpha(\overline S_h)$ ($\alpha\in (0,1]$), it readily has $f(\mathbf x):=(q(\mathbf x)-1)v(\mathbf x) \in C^{\alpha} (\overline{S_h})$. Hence we  have the expansion of $f(\mathbf x),\ \eta$ and $v(\mathbf x)$ near the origin as follows
       \begin{equation}\label{eq:f expan 2}
       	\begin{aligned}
       	f(\mathbf x)&= f(\mathbf 0)+\delta f,\quad \vert \delta f\vert\leq \|f\|_{C^\alpha}\vert \mathbf x \vert ^\alpha\\
       	\eta&= \eta(\mathbf 0)+\delta \eta,\quad \vert \delta \eta\vert\leq \|\eta\|_{C^\alpha}\vert \mathbf x\vert ^\alpha\\
       	v(\mathbf x)&= v(\mathbf 0)+\delta v,\quad \vert \delta v\vert\leq \|v\|_{C^\alpha}\vert \mathbf x\vert ^\alpha
       	\end{aligned}
       \end{equation}
      Plugging \eqref{eq:f expan 2} into the integral  identity  \eqref{1eq:green}, it yields  that
          \begin{align}
         \eta(\mathbf 0)v(\mathbf 0)\int_{\Gamma_h^{\pm}}e^{\rho\cdot \mathbf x}\mathrm d\sigma&=f(\mathbf 0)\int_{S_h}u_0\mathrm d\mathbf x+\int_{S_h}\delta fu_0\mathrm d\mathbf x
         +\eta(\mathbf 0)v(\mathbf 0)\int_{\Gamma_h^{\pm}}\psi(\mathbf x)e^{\eta \cdot \mathbf x}\mathrm d\sigma\notag \\
         &+\eta(\mathbf 0)\int_{\Gamma_h^{\pm}}\delta v u_0\mathrm d\sigma
         +v(\mathbf 0)\int_{\Gamma_h^{\pm}}\delta \eta u_0\mathrm d\sigma
         +\int_{\Gamma_h^\pm}\delta v\delta\eta u_0\mathrm d\sigma\notag \\
         &-\int_{\Lambda_h}(w-v)\partial_\nu u_0\mathrm d\sigma+\int_{\Lambda_h}u_0\partial_\nu (w-v)\mathrm d\sigma. \label{eq:int 349}
         \end{align}
         By adopting  similar asymptotic  analysis  for each integrals in \eqref{eq:int 349}   with respect to  the  parameter $\tau$ as in the proof of Theorem \ref{thm:2D}, and letting $\tau \rightarrow \infty$, we can prove  Theorem \ref{2D:delta}.
       \end{proof}

       \subsection{Proof of Corollary \ref{cor1} }\label{subsec:32}

     Next, we give the proof of  Corollary \ref{cor1} regarding  the vanishing property of transmission eigenfunctions to \eqref{3eq:treta0}  near a convex planar corner under two regularity conditions described in Corollary \ref{cor1}. Since the proof of the statement (b) in Corollary \ref{cor1}  can be obtained by using the similar asymptotic analysis for proving  Corollary \ref{cor1} (a), we omit it here. In order to prove  the statement  (a) in Corollary \ref{cor1}, we give the following proposition which is obtained by slightly modifying the proof of Proposition \ref{2prop:tau2}.


\begin{prop}\label{prop:tau2}
	Let $S_{h}$ and $\eta$ be defined in (\ref{1eq:not}) and (\ref{1eq:eta}), respectively, where $ \mathbf d$ satisfies   \eqref{2eq:zeta}. Then we have
	\begin{equation}\label{1eq:tau}
	\left\vert \int_{S_h}e^{\eta \cdot \mathbf x}\mathrm d \mathbf x\right \vert \geq  \frac{\widetilde {C_{S_h }}}{\tau^2}-\mathcal O\left(\frac{1}{\tau}e^{-\frac{1}{2}\zeta h \tau}\right),
	\end{equation}
	for sufficiently large $\tau$, where  $\widetilde {C_{S_h }}$ is a positive number only depending on the opening angle $\theta_M-\theta_m$ of $\mathcal K$ and $\zeta$.
\end{prop}

\begin{proof}
	Using polar coordinates transformation and \eqref{1eq:gamma} in Proposition \ref{1prop:gamma}, we have
	\begin{equation}\notag
	\begin{aligned}
	\int_{S_h} e^{\rho \cdot \mathbf x}\mathrm d \mathbf x
	&=\int_{\theta_{m}}^{\theta_M}\left[\frac{\Gamma(2)}{(\tau(\mathbf d+\mathrm i \mathbf d^\perp)\cdot \mathbf {\hat x} )^2}-I_{\sf R}\right]\mathrm d\theta\\
	&=\frac{\Gamma(2)}{\tau^2}\int_{\theta_{m}}^{\theta_M}\frac{1}{\left(\mathbf d\cdot \mathbf{\hat x}+\mathrm i \mathbf{d}^\perp \cdot \mathbf{\hat x} \right)^2}\mathrm d \theta-\int_{\theta_{m}}^{\theta_M}I_{\sf{R}}\mathrm d\theta,
	\end{aligned}
	\end{equation}
	where $I_{\sf R}= \int_{h}^{\infty}e^{-\tau (\mathbf d +\mathrm i\mathbf d)\cdot \hat{\mathbf x}r}r\mathrm d r$.
	Hence,  it can be directly calculated that
	\begin{equation}\notag
	\int_{\theta_{m}}^{\theta_M}\frac{1}{\left(\mathbf d\cdot \mathbf{\hat x}+\mathrm i \mathbf{d}^\perp \cdot \mathbf{\hat x} \right)^2}\mathrm d \theta
	\geq \frac{\theta_M-\theta_m }{2}
	\end{equation}
	by using the integral mean value theorem.
	
	With the help of Proposition \ref{1prop:gamma}, for sufficiently large $\tau$, we have the following integral inequality
	\begin{equation}\label{1eq:eta0tau2}
	\begin{aligned}
	\left\vert \int_{S_h} e^{\rho \cdot \mathbf x}\mathrm d \mathbf x  \right\vert
	&\geq \frac{\Gamma(2)(\theta_{M}-\theta_{m})}{\tau^2} \frac{1}{\left\vert \mathbf d\cdot \mathbf x( \theta_{\xi})+ \mathrm i \mathbf{d}^\perp \cdot \mathbf{\hat x} (\theta_{\xi}) \right\vert^2}- \left\vert \int_{\theta_{m}}^{\theta_M} I_{\sf{R}} \mathrm d\mathbf \theta \right\vert\\
	&\geq \frac{\Gamma(2)}{(\theta_{M}-\theta_{m})}\frac{1}{\left(\mathbf d\cdot \mathbf{\hat x}(\theta_{\xi}) \right)^2+ \left(\mathbf{d}^\perp \cdot \mathbf{\hat x}(\theta_{\xi})  \right)^2}-\int_{\theta_{m}}^{\theta_M} \vert I_{\sf{R}} \vert \mathrm d\theta\\
	&\geq  \frac{\widetilde {C_{S_h }}}{\tau^2}-\frac{1}{\tau}e^{-\frac{1}{2}\zeta h \tau},
	\end{aligned}
	\end{equation}
	by using \eqref{2eq:zeta}.
\end{proof}

\begin{proof}[\bf {Proof of Corollary \ref{cor1}(a)}]\label{proof:holder2}
	Similar to the proof of Theorem \ref{thm:2D}, we have the following integral identity according to \eqref{1eq:green}  by noting $\eta \equiv 0$ on $\Gamma_h^\pm$,
	\begin{equation}\label{2eq:eta0I}
	k^2f_j(0)\int_{S_h}e^{\rho \cdot \mathbf x}\mathrm d \mathbf x=I_1+I_2+I_3+J_1+J_2,
	\end{equation}
	where
	\begin{align*}
	I_1&=-k^2\int_{S_h}(q-1)(v-v_j)u_0\mathrm d \mathbf x
	,\quad I_2=-k^2\int_{S_h}\delta f_ju_0\mathrm d\mathbf x, \\
	I_3&=-k^2f_j(\mathbf 0)\int_{S_h}e^{\rho \cdot \mathbf x}\psi(\mathbf x)\mathrm d\mathbf x,
	\end{align*}
	and $J_1$, $J_2$ are defined in \eqref{2eq:chaifen}, respectively.


 By the Sobolev embedding theorem and $q\in H^2(\overline S_{h})$, we have $q\in C^\alpha(\overline S_{h} )$, where $\alpha =1$.   Combining \eqref{2eq:eta0I} with \eqref{2eq:I1}, \eqref{2eq:I2} and \eqref{2eq:I33}, we can  deduce that
 \begin{equation}\label{2eq:fj0}
 \begin{aligned}
 k^2\left[ \frac{\widetilde{C_{S_h }}}{\tau^2}-\frac{1}{\tau}e^{-\frac{1}{2}\zeta h \tau}\right] &\vert f_j(\mathbf 0)\vert\lesssim
 j^{-\beta}\left[\left(\frac{1}{\tau^2}+\frac{1}{\tau}e^{-\zeta h\tau}\right)^{\frac{1}{2}}+\left(\frac{1}{\tau^2}+\frac{1}{\tau}e^{-2\zeta h\tau}\right)^{\frac{1}{4}}\tau^{-\frac{2}{3}} \right]\\
 &+j^{\gamma }\left[\tau^{-(\alpha+\frac{29}{12})}+\left(\frac{1}{\tau^{\alpha+2}}+\frac{1}{\tau}e^{-\frac{1}{2}\zeta h\tau}\right)\right]\\
 &+\tau^{-\frac{29}{12}}
 +(1+\tau)(1+\tau^{-\frac{2}{3}})e^{-\zeta h\tau}
 \end{aligned}
 \end{equation}
 as $\tau \to \infty$.  Multiplying  $\tau^2$ on the both sides of \eqref{2eq:fj0}, using the assumption \eqref{1eq:herg}, by letting $\tau=j^s$, it is easy to see that
 \begin{equation}\label{2eq:j}
 k^2\widetilde {C_{S_h }}\vert f_j(\mathbf 0)\vert \lesssim j^{-\beta +s}+j^{\gamma - \alpha s}.
 \end{equation}
And under the assumptions $\gamma/\alpha<\beta$, we choose $s\in(\gamma/\alpha,\beta)$.  Letting $j\to \infty$ in \eqref{2eq:j}, we obtain that
$$\vert f_j(\mathbf 0)\vert =0.$$
Since $q(\mathbf 0)\not =1$, we finish the proof of this corollary.
\end{proof}

       \section{Vanishing of transmission eigenfunctions near a convex conic corner or polyhedral corner}\label{sec:3D}

       In this section, we study the vanishing of eigenfunctions near a corner in $\mathbb R^3$ respectively, where  the  corner in $\mathbb R^3$ could  be  a convex conic  corner or polyhedral corner. Let us first  introduce the corresponding geometrical setup for our study.
       For a given point $\mathbf{x}_0\in\mathbb{R}^3$, let $\mathbf{v}_0=\mathbf y_0-\mathbf x_0$ where $\mathbf y_0\in\mathbb{R}^3$ is fixed.  Hence
       \begin{equation}\label{eq:cone1}
       {\mathcal C}={\mathcal C}_{\mathbf{x}_0,\theta_0}:= \left\{\mathbf{y} \in \mathbb R^3 ~|~0\leqslant \angle(\mathbf y-\mathbf{x}_0,\mathbf{v}_0)\leqslant \theta_0\right \}\ (\theta_0 \in(0,\pi/2))
       \end{equation}
       is a strictly convex conic cone with the apex $\mathbf x_0$ and an opening angle $2\theta_0 \in(0,\pi)$  in $\mathbb R^3$. Here $\mathbf v_0$ is referred to be the axis of $\mathcal C_{\mathbf x_0,\theta_0}$. Specifically,  when $\mathbf{x}_0=\mathbf 0$,  $\mathbf{v}_0=(0,0,1)^\top$, we write $\mathcal C_{\mathbf x_0,\theta_0 }$ as  $\mathcal{C}_{\theta_0}$. Define the truncated conic cone $\mathcal C^{h}:=\mathcal C^{h} _{ \mathbf 0}$ as
       \begin{equation}\label{eq:cone2}
       \mathcal C^{h}:=\mathcal C_{\theta_0}\cap B_{h},\quad \Gamma_h=\partial \mathcal C\cap B_h,\quad \Lambda_h=\mathcal C \cap \partial B_h,
       \end{equation}
       where $B_{h}$ is an open ball centered at $\mathbf 0$ with the radius $h\in \mathbb R_+$.

   Assume that $\mathcal K_{\mathbf x_0;\mathbf e_1,\ldots, \mathbf e_\ell}$ is a polyhedral cone with the apex $\mathbf x_0$ and edges $\mathbf e_j$ ($j=1,\ldots, \ell$, $\ell\geq 3$. Throughout of this paper we always suppose that  $  \mathcal K_{\mathbf x_0;\mathbf e_1,\ldots, \mathbf e_\ell}$ is strictly convex, which implies that   it can be fitted  into a conic  cone $\mathcal C_{\mathbf x_0, \theta_0}$ with the opening angle $\theta_0\in (0,\pi/2)$, where $\mathcal C_{\mathbf x_0, \theta_0}$  is defined in \eqref{eq:cone1}.  Without loss of generality, we  assume that  the axis of $\mathcal C_{\mathbf x_0, \theta_0}$  coincides with $x_3^+$ and $\mathbf x_0=\mathbf 0$. Given a constant $h\in \mathbb R_+$, we define the truncated polyhedral corner $\mathcal K_{\mathbf x_0}^{h}$ as
     \begin{equation}\label{eq:kr0}
       \mathcal K_{\mathbf x_0}^{h}=\mathcal K_{\mathbf{x_0}; {\mathbf e_1},\ldots {\mathbf e_\ell }}\cap B_{h}.
	\end{equation}
For convenience, we have a similar geometry setup with \eqref{eq:cone2} as
 \begin{equation}\label{eq:cone3}
 \mathcal K^h=\mathcal K_{\mathbf 0}^h,\ \Gamma_h=\partial \mathcal K\cap B_h, \ \Lambda _h=\mathcal K\cap \partial B_h.
 \end{equation}

       The following theorem states  that the transmission eigenfunctions to \eqref{0eq:tr} must vanish at a conic corner  if they have $H^1$ regularity and $v$ can be approximated by a sequence of Herglotz wave functions near the underlying conic corner with certain properties, where the detailed proof is postponed to Subsection \ref{subsec:41}.

       \begin{thm}\label{3thm:v0}
       	Let $\Omega$ is a bounded Lipschitz domain with a connected complement  and  $v,w\in H^1(\Omega)$ be a pair of transmission eigenfunctions to (\ref{0eq:tr}) associated with $k\in \mathbb R_+$. Assume that $\mathbf 0\in \Gamma \subset  \partial \Omega$ such that $\Omega\cap B_h=\mathcal C\cap B_h=\mathcal C^h$, where $\mathcal C$ is defined by (\ref{eq:cone1}) and $h\in \mathbb R_+$ is sufficient small such that $q\in H^2( \overline{\mathcal C^h})$ and $\eta\in C^{\alpha_1} (\overline{\Gamma_h})$, where  $\alpha_1\in (0,1)$. If the following conditions are fulfilled:
       	\begin{itemize}
       	  \item[(a)]for any give positive constants $\beta$ and $\gamma$ satisfying
       	  \begin{equation}\label{eq:assump2}
       	 \gamma<\frac{10}{11 }\alpha\beta,\quad \alpha=\min\{\alpha_1,1/2 \},
       	  \end{equation}
       	  the transmission eigenfunction $v$ can be approximated in $H^1(\mathcal C^h)$ by Herglotz functions
       	  \begin{equation}\label{2eq:herg}
       	  v_j=\int_{\mathbb S^2}e^{\mathrm ik\xi\cdot \mathbf x}g_j(\xi)\mathrm d\xi,\quad \xi\in \mathbb S^2, j=1,2,\cdots,
       	  \end{equation}
       	  with the kernels $g_j$ satisfying the approximation property
       	  \begin{equation}\label{3eq:kernel}
       	  \|v-v_j\|_{H^1(\mathcal C^h)}\leq j^{-\beta},\quad \|g_j\|_{L^2(\mathcal C^h)}\leq j^\gamma ;
       	  \end{equation}
       	
       	\item[(b)]the function $\eta$ dose not vanish at the apex $\mathbf 0$ of $\mathcal C^h$;
       	\end{itemize}
       	then one has
       	\begin{equation}\label{3eq:delv0}
       	\lim_{\lambda\to \infty}\frac{1}{m(B(\mathbf 0,\lambda)\cap\Omega)}\int_{B(\mathbf 0,\lambda)\cap \Omega}\vert v(\mathbf x)\vert \mathrm d \mathbf x=0,
       	\end{equation}
       	where $m(B(\mathbf 0,\lambda)\cap \Omega)$ is the area of $B(\mathbf 0,\lambda)\cap \Omega$.
       \end{thm}

  As remarked earlier, the Herglotz approximation property in \eqref{3eq:kernel} characterises a regularity of $v$ weaker than the H\"older continuity (cf. \cite{LT}).      In the following theorem, if a stronger H\"older regularity condition near a conic corner on the transmission eigenfunction $v$ to \eqref{0eq:tr} is satisfied, we also have the vanishing characterization of the corresponding transmission eigenfunction $v$. Namely, when $v$ is H\"older continuous near the underlying circular corner, we show that  it  must  vanish at the apex of the conic  corner. The proof can be obtained  by modifying the corresponding proof of Theorem \ref{3thm:v0} directly as for the two dimensional case, which is omitted.

        \begin{thm}\label{thm:3alpha}
        	Let $v\in H^1(\Omega)$ and $w\in H^{1}(\Omega)$ be a pair of transmission  eigenfunctions to \eqref{0eq:tr} associated with $k\in \mathbb R_+$. Assume that the Lipschitz domain $\Omega\subset \mathbb R^3$ with $\mathbf 0 \in \partial \Omega$ contains a  conic corner $\Omega\cap B_h=\mathcal C\cap B_h=\mathcal C^h$, such that $v\in C^{\alpha}(\overline{\mathcal C^h})$, $q\in H^2(\overline {\mathcal C^h})$ and $\eta \in C^\alpha (\overline{\Gamma _h}) $ for $0<\alpha <1$, where $B_h,\ \Gamma_h $ and $\mathcal C^h$ are defined in (\ref{eq:cone2}). If  $\eta(\mathbf 0)\not=0$, where $\mathbf 0$ is the apex of $\mathcal C^h$,
        	then one has
        	\begin{equation}\label{3eq:v0}
        	v(\mathbf 0)=0.
        	\end{equation}
        \end{thm}

Consider a  cuboid corner $\mathcal K_{\mathbf x_0;\mathbf e_1, \mathbf e_2, \mathbf e_3}$ defined by \eqref{eq:kr0}. In Theorem \ref{3:cubiod}, we show that the transmission eigenfunctions to \eqref{0eq:tr} vanish at the  cuboid corner $\mathcal K_{\mathbf x_0;\mathbf e_1, \mathbf e_2, \mathbf e_3}$ when they are H\"older continuous at the corner point. The proof of Theorem \ref{3:cubiod} can be found  in Subsection \ref{subsec:42}. Since $\Delta$  is invariant under rigid motion, we assume that the apex $\mathbf x_0$ of $\mathcal K_{\mathbf x_0;\mathbf e_1, \mathbf e_2, \mathbf e_3}$ coincides with  the origin, and the edges of $\mathcal K_{\mathbf x_0;\mathbf e_1, \mathbf e_2, \mathbf e_3}$  satisfy $\mathbf e_1=(1,0,0)^\top  $, $\mathbf e_2=(0,1,0)^\top  $ and $\mathbf e_3=(0,0,1)^\top  $.

	\begin{thm}\label{3:cubiod}
		Let $v\in H^1{(\Omega}),\ w\in H^1{(\Omega)}$ be a pair of transmission eigenfunctions of \eqref{0eq:tr} with $k>0$. Assume that Lipschitz domain $\Omega \subset \mathbb R^3$ with $\mathbf 0 \in \Gamma \subset \partial \Omega$ contains a cuboid corner $\Omega \cap B_h =\mathcal K \cap B_h =\mathcal K^h,$ such that $v\in C^\alpha (\overline{\mathcal K^h}) $, $q\in H^2(\overline{\mathcal K^h})$ and  $\eta \in C^\alpha (\overline{\Gamma _h})$  for $0<\alpha < 1$, where $\mathcal K^h$  and  $\Gamma _h$ are defined in \eqref{eq:kr0}. If   $\eta (\mathbf 0)\not=0$, then
		$$v(\mathbf 0)=0.$$
	\end{thm}
	
\begin{rem}
	Consider the classical  transmission eigenvalue  problem \eqref{3eq:treta0} in $\mathbb R^3$, namely $\eta\equiv  0$ on $\Gamma$ in \eqref{0eq:tr},  when the underlying domain $\Omega$ of \eqref{3eq:treta0} has  a  cuboid corner $\mathcal K_{\mathbf x_0;\mathbf e_1, \mathbf e_2, \mathbf e_3}$, if the corresponding  potential $q$ has $\alpha$-H\"older continuity  regularity for $\alpha >\frac{1}{4}$ near the cuboid corner (cf. \cite[Definition 2.2 and Theorem 3.2]{BL2017}), then the transmission eigenfunction $v$ must vanish near the corner.  Compared with the results in \cite{BL2017}, the vanishing property of transmission eigenfunctions to \eqref{0eq:tr} near the underlying cuboid corner holds under a general scenario. Namely, the assumption in Theorem \ref{3:cubiod} only needs $q$ fulfills $H^2$ regularity, $v$ and boundary parameter $\eta$ are H\"older continuous near $\mathbf x_0$, where $\eta(\mathbf x_0) \neq 0$.

\end{rem}

In the following two corollaries, we consider the classical transmission eigenvalue problem \eqref{3eq:treta0}, namely $\eta\equiv  0$ on $\Gamma$ in \eqref{0eq:tr}, where the domain $\Omega$ contains a conic or polyhedral corner. The proof of Corollary \ref{cor2} is postponed in Subsection \ref{subsec:43}.

 
\begin{cor}\label{cor2}
	Let $\Omega$ is a bounded Lipschitz domain with a connected complement  and  $v,w\in H^1(\Omega)$ be a pair of transmission eigenfunctions to \eqref{3eq:treta0}  associated with $k\in \mathbb R_+$. Assume that $\mathbf 0\in \Gamma\subset \partial \Omega$ such that $\Omega\cap B_h=\mathcal C\cap B_h=\mathcal C^h$, where $\mathcal C$ is defined by (\ref{eq:cone1}) and $h\in \mathbb R_+$ is sufficient small such that $q\in H^2( \overline{\mathcal C^h})$ and $q(\mathbf 0)\neq 1$.
	\begin{itemize}
		\item[(a)]For any give positive constants $\beta$ and $\gamma$ satisfying
		$
		\gamma <\frac{20}{37}\alpha\beta,
		$
		 if the transmission eigenfunction $v$ can be approximated in $H^1(\mathcal C^h)$ by Herglotz wave functions $v_j$ defined by \eqref{2eq:herg}
		with the kernels $g_j$ satisfying the approximation property
		\eqref{3eq:kernel}, then we have the vanishing of the transmission eigenfunction $v$ near $\mathcal C^h$ in the sense of \eqref{3eq:delv0}.
		\item[(b)]If $v\in C^\alpha (\overline{\mathcal C ^h})$ with $\alpha\in (0,1)$, then one has $v(\mathbf 0)=0$.
	\end{itemize}
\end{cor}

In the  subsequent  corollary, we  consider the case that $\Omega$ contains a polyhedral corner $\mathcal K^{h}$ defined by \eqref{eq:kr0}. When the transmission eigenfunction $v$ to  \eqref{3eq:treta0} satisfies two  regularity assumptions, we can establish the similar geometrical characterization of $v$ near the polyhedral corner. The proofs are similar to the counterpart  of Theorem \ref{thm:3alpha} and Corollary \ref{cor2}, where we only need to use the asymptotic analysis  \cite[Lemma 2.2]{BLX2020}  with respect to the parameter in the corresponding  CGO solution  introduced the following subsection. Hence  we omit its proof.

\begin{cor}\label{thm:poly}
	Let $\Omega$ is a bounded Lipschitz domain with a connected complement  and  $v,w\in H^1(\Omega)$ be a pair of transmission eigenfunctions to \eqref{3eq:treta0}  associated with $k\in \mathbb R_+$. Assume that $\mathbf 0\in \Gamma \subset  \partial \Omega$ such that $\Omega\cap B_h=\mathcal K_{\mathbf 0;\mathbf e_1,\cdots,\mathbf e_{\ell}}\cap B_{h}=\mathcal K^{h}$, where $\mathcal K^{h}$ is defined by \eqref{eq:kr0} and $h\in \mathbb R_+$ is sufficiently small such that $q\in H^2( \overline{\mathcal K^h})$  and $q(\mathbf 0)\neq 1$.
	\begin{itemize}
		\item[(a)]
		For any give positive constants $\beta$ and $\gamma$ satisfying $\gamma <\frac{20}{37}\alpha\beta$, if
		 the transmission eigenfunction $v$ can be approximated in $H^1(\mathcal C^h)$ by Herglotz wave functions $v_j$ defined by \eqref{2eq:herg}
		with the kernels $g_j$ satisfying the approximation property \eqref{3eq:kernel}, then we have the  vanishing property of  $v$ near  $\mathcal K^{h}$ in the  sense of \eqref{3eq:delv0}.
		\item[(b)] If $v\in C^{\alpha}(\overline{ \mathcal K^{h}})$ with $\alpha\in (0,1)$, then one  has $v(\mathbf 0)=0$.
	\end{itemize}
	
\end{cor}

   \subsection{Proof of Theorem \ref{3thm:v0}}\label{subsec:41}

   Since the conic cone $\mathcal C$ defined by  \eqref{eq:cone1}  is strictly convex, for any given positive constant $\zeta$,  we define $\mathcal C_\zeta$ as the open set of $\mathbb S^2$ which is composed by all unit directions $\mathbf d\in \mathbb S^2$ satisfying that
       \begin{equation}\label{3eq:3zeta}
       \mathbf d\cdot \mathbf {\hat{x}}>\zeta>0,\quad for\  all \ \mathbf {\hat{x}}\in \mathcal C\cap \mathbb S^2.
       \end{equation}
       Through out of this subsection, we always assume that the unit vector $\mathbf d$ in the form of the CGO solution $u_0$ given by (\ref{1eq:cgo}) satisfies (\ref{3eq:3zeta}). In order to prove Theorem \ref{3thm:v0}, we need several key propositions and lemmas in the following.


       \begin{prop}\label{3prop:eta}
        Let $\Gamma_{h}$ and $\rho$ be defined in (\ref{eq:cone2}) and (\ref{1eq:eta}), respectively. Then we have
        \begin{equation}\label{3eq:eta}
        \left\vert \int_{\Gamma_{h}} e^{\rho \cdot \mathbf x}\mathrm d\sigma\right\vert \geq \frac{C_{\mathcal C^h}}{\tau^2}-\mathcal O\left(\frac{1}{\tau}e^{-\frac{1}{2}\zeta h \tau}\right),
        \end{equation}
        for sufficiently large $\tau$, where $C_{\mathcal C^h}$ is a positive number only depending on the opening angle $\theta_0$ of $\mathcal C$ and $\zeta$.
       \end{prop}
       \begin{proof}
       	Using polar coordinates transformation and the  mean value theorem for integrals, we have
       	\begin{equation}\label{eq:estconic}
       	\begin{aligned}
       	\int_{\Gamma_h}e^{\rho \cdot \mathbf x}\mathrm d\sigma
       	&=\sin\theta_0\frac{2\pi\Gamma(2)}{\tau^2}\frac{1}{((\mathbf d+\mathrm i\mathbf d^{\perp})\cdot \mathbf{ \hat{x}}(\theta_0,\varphi_\xi))^2}-\sin\theta_0\int_{0}^{2\pi}I_{R}\mathrm d\varphi,
       	 \end{aligned}
       	\end{equation}
       	where $I_{R}= \int_{h}^{\infty}e^{-\tau (\mathbf d +\mathrm i\mathbf d)\cdot \hat{\mathbf x}r}r\mathrm d r$.
       	Furthermore, for sufficiently large $\tau$, it is ready  to know that
       	$$\frac{1}{\tau^2}-\frac{1}{\tau}e^{-\frac{1}{2}\zeta h \tau}>0.$$
       	Hence, by virtue  of \eqref{3eq:3zeta} and  Proposition \ref{1prop:gamma}, we have the following integral inequality
       	\begin{equation}\notag
       	\begin{aligned}
       	\left\vert \int_{\Gamma _h}e^{\rho \cdot \mathbf x}\mathrm d\sigma \right \vert
       	&\geq \sin\theta_0 \frac{2\pi \Gamma (2)}{\tau^2}\frac{1}{(\mathbf d\cdot \mathbf{\hat x}(\theta_{0},\varphi_{\xi}))^2+(\mathbf d^\perp \cdot \mathbf{\hat x}(\theta_{0},\varphi_\xi))^2}-\sin\theta_0\int_{0}^{2\pi}|I_R|\mathrm d\varphi\\
       		&\geq \frac{C_{\mathcal C^h}}{\tau^2}-\mathcal O(\frac{1}{\tau}e^{-\frac{1}{2}\zeta h \tau}),
       		\end{aligned}
       	\end{equation}
   which completes the proof of this proposition.
       \end{proof}
        Similar to Proposition \ref{1prop:enorm}, the following proposition can be obtained by directly verifications.
        \begin{prop}\label{2prop:enorm}
        	Let $\mathcal C^h$ be defined by  \eqref{eq:cone2}. For any given $t>0$, it yields that
        	\begin{equation}\label{2eq:enorm}
        	\begin{aligned}
        	\| e^{\rho \cdot \mathbf {x}} \|_{L^t(\mathcal C^h)} &\leq C \left(\frac{1}{\tau^3}+\frac{1}{\tau}e^{-\frac{t}{2}\zeta h\tau}\right)^{\frac{1}{t}},\\
        	\| e^{\rho \cdot \mathbf {x}} \|_{L^t(\Gamma _h)} &\leq C \left(\frac{1}{\tau^2}+\frac{1}{\tau}e^{-\frac{t}{2}\zeta h\tau}\right)^{\frac{1}{t}},
        	\end{aligned}
        	\end{equation}
        	as $\tau\rightarrow \infty$,  where $\rho$ is defined in (\ref{1eq:eta}) and $C$ is a positive constant only depending on $t,\zeta$.
        	
        	\end{prop}

       \begin{lem}
       	Under the same setup of Theorem \ref{3thm:v0}, let the CGO solution $u_0$ be defined by (\ref{1eq:cgo}). We also denote $u=w-v$, where $(v,w)$ is a pair of transmission eigenfunctions of \eqref{0eq:tr} associated with $k \in \mathbb R_+$.  Then it holds that
       	\begin{equation}\label{eq:pde sys 3}
       	\begin{cases}
       	&\Delta u_0 +k^2qu_0=0\hspace*{1.7cm}\mbox{in}\quad  {\mathcal C}^h,\\
       	&\Delta u +k^2qu=k^2(1-q)v \hspace*{0.5cm} \mbox{in}\quad  {\mathcal C}^h,\\
       	&u=0,\quad \partial_{\nu}u=0  \hspace*{1.7cm} \mbox{on}\quad \Gamma_h,
       	\end{cases}
       	\end{equation}
       	where $\mathcal C^h$ and $\Gamma_h$ are defined by  \eqref{eq:cone2},
       	and
       	\begin{equation}\label{3eq:psinorm}
       	\| \psi(\mathbf x) \|_{H^{1,8}} 
       	= \mathcal O(\tau^{-\frac{2}{5}}),
       	\end{equation}
       	where $\psi$ and $\tau$ are defined in \eqref{1eq:cgo}.
       \end{lem}
       \begin{proof}
       The proof of this lemma is similar to the one of Lemma \ref{lem:31}. By using the assumption $q\in H^2(\mathcal C^h)$ and Sobolev extension  property, suppose that $\tilde q$ is a Sobolev extension of $q$ in $\mathbb R^3$, it yields that $\tilde q\in H^{2}$, hence we have $\tilde{q}\in H^{1,1+\epsilon_0}$.
       	Therefore $\tilde q$ satisfies the assumption in Proposition \ref{1pro:nor}. Let $\tilde p=\frac{120}{79}$ and $\epsilon_0=\frac{7}{8}$, thus we know that $p=8$ by (\ref{eq:p til cond}). By (\ref{1nor:psi}), it yields that (\ref{3eq:psinorm}).
       \end{proof}

           \begin{lem}\label{3lem:u0est}
       	Let $\Lambda_h$ and $\mathcal C^h$ be defined in (\ref{eq:cone2}). Recall that $u_0(\mathbf x)$ is given by (\ref{1eq:cgo}). Then $u_0(\mathbf x)\in H^1({\mathcal  C}^h)$ and it holds that
       	\begin{subequations}
       		\begin{align}
       		\|u_0(\mathbf x)\|_{L^2(\Lambda_h)}&\lesssim \left(1+\tau^{-\frac{2}{5}}\right)e^{-\zeta h\tau},\label{3eq:u0lambda} \\
       		\label{3eq:nablau0}
       		\|\nabla u_0(\mathbf x)\|_{L^2(\Lambda_h)}&\lesssim (1+\tau) \left(1+\tau^{-\frac{2}{5}}\right)e^{-\zeta h\tau},\\
       		\label{3eq:u0alpha}
       		\int_{\mathcal C^h}\vert \mathbf x\vert^{\alpha}\vert u_0(\mathbf x)\vert\rm d\mathbf x
       		&\lesssim \tau^{-(\alpha+\frac{121}{40})}+\left(\frac{1}{\tau^{\alpha+3}}+\frac{1}{\tau}e^{-\frac{1}{2}\zeta h\tau}\right),\\
       		 \int_{\Gamma _h} \vert \mathbf x\vert ^\alpha \vert u_0(\mathbf x)\vert \rm d \sigma
       		&\lesssim \tau^{-(\alpha +\frac{43}{20})}+\left(\frac{1}{\tau^{\alpha +2}} +\frac{1}{\tau}e^{-\frac{1}{2}\zeta h\tau}\right)\label{3eq:u0alphagam}
       		\end{align}
       	\end{subequations}
       		as $\tau \rightarrow \infty$, where    $\zeta$  is defined  in \eqref{2eq:zeta} and $\alpha \in (0,1)$.
       		%
       		%
       		%
       	\end{lem}

        \begin{proof}
        	Using (\ref{3eq:psinorm}) and Lemma \ref{lem:trace thm} about the trace theorem, it yields that
        	\begin{equation}\notag 
        	\|\psi(\mathbf x)\|_{L^{4}(\Lambda_h)}
	\leq C \|\psi(\mathbf x)\|_{H^{\frac{7}{8},8}(\Lambda_h)}\leq C\|\psi(\mathbf x)\|_{H^{1,8}(\mathcal C^h)}= O(\tau^{-\frac{2}{5}}),
        	       	\end{equation}
		as  $\tau \rightarrow \infty$.  
      By using polar coordinates transformation, (\ref{1eq:Ir}), and (\ref{3eq:3zeta}), one can derive that
        	\begin{equation}\label{3eq:etabound}
        	\|e^{\rho \cdot \mathbf x}\|_{L^t(\Lambda_h)}\lesssim e^{-\zeta h\tau},
        	\end{equation}
        	where $\rho$ is defined in (\ref{1eq:eta}) and $t$ is a positive constant.
        	
      Due to polar coordinates transformation, (\ref{3eq:etabound}), \eqref{3eq:psinorm} and H\" older  inequality, it can be calculated that
        	\begin{equation}\label{3eq:I42}
        	\begin{aligned}
        	\|u_{0}\|_{ L^{2}(\Lambda_{h})}
        	&\leq\|e^{\rho \cdot \mathbf x}\|_{ L^{2}(\Lambda_{h})}+\|e^{\rho \cdot \mathbf x}\|_{ L^{4}(\Lambda_{h})}\|\psi(\mathbf x)\|_{ L^{4}(\Lambda_{h})}\\
        	&\lesssim \left(1+\tau^{-\frac{2}{5}}\right)e^{-\zeta h\tau}, \quad \mbox{as  $\tau \rightarrow \infty$.  }
        	\end{aligned}
        	\end{equation}
        	By virtue of  Cauchy-Schwarz  inequality, (\ref{3eq:u0lambda}) and Proposition \ref{1prop:gamma}, we can deduce  that
        	\begin{equation}\label{3eq:I43}
        	\begin{aligned}
        	\|\nabla u_{0}\|_{L^{2}(\Lambda_{h})}
        	&\leq \sqrt{2}\tau \|u_{0}\|_{ L^{2}(\Lambda_{h})}+\|e^{\rho \cdot \mathbf x}\|_{ L^{4}(\Lambda_{h})} \|\nabla \psi(\mathbf x)\|_{ L^{4}(\Lambda_{h})}\\
        	&\lesssim(1+\tau)(1+\tau^{-\frac{2}{5}})e^{-\zeta h\tau},\quad \mbox{as  $\tau \rightarrow \infty$. }
        	\end{aligned}
        	\end{equation}
        It is clear that we can get the following integral inequality,
        	\begin{equation}\label{3eq:xalpha}
        	\begin{aligned}
        	\int_{\mathcal C^h}\vert \mathbf x\vert^{\alpha}\vert u_{0} \vert\mathrm{d}\mathbf x
        	&\leq\int_{\mathcal C^h}\vert \mathbf x\vert^{\alpha}\vert e^{\rho \cdot \mathbf{x}} \vert \mathrm{d}\mathbf x+\int_{\mathcal C^h}(\vert \mathbf x\vert^{\alpha}\vert e^{\rho \cdot \mathbf{x}}\vert)(\vert \psi(\mathbf x)\vert)\mathrm{d}\mathbf x.
        	\end{aligned}
        	\end{equation}
       By virtue of  polar coordinates transformation and Proposition \ref{1prop:gamma}, it arrives that
        	\begin{equation}\label{3eq:I22}
        	\begin{aligned}
        	\int_{\mathcal C^h}\vert\mathbf x\vert^{\alpha}e^{-\zeta \tau \vert \mathbf x\vert}\mathrm{d}\mathbf x
        	\lesssim \frac{1}{\tau^{\alpha+3}}+\frac{1}{\tau}e^{-\frac{1}{2}\zeta h\tau },
        	\end{aligned}
        	\end{equation}
     as  $\tau \rightarrow \infty$.   Next, letting $\mathbf y=\tau \mathbf x$,  using Cauchy-Schwarz inequality and H$\ddot {\mathrm o}$lder inequality, it arrives that
      \begin{align}
      \int_{\mathcal C^h}(\vert \mathbf x\vert^{\alpha}\vert e^{\rho \cdot \mathbf{x}}\vert)(\vert \psi(\mathbf x)\vert)\mathrm{d}\mathbf x
      &\leq
      \frac{1}{\tau^{\alpha +3}}\int_{\mathcal C}\vert \mathbf y\vert^\alpha\vert e^{-\mathbf d\cdot \mathbf y}\vert \left\vert \psi\left(\frac{\mathbf y}{\tau}\right)\right\vert \mathrm d\mathbf y\notag \\
      &\leq\frac{1}{\tau^{\alpha +3}}\|\vert \mathbf y\vert^\alpha\vert e^{-\mathbf d\cdot \mathbf y}\vert \|_{L^{\frac{8}{7}}(\mathcal C)}\left\|\psi\left(\frac{\mathbf y}{\tau}\right)\right\|_{L^{8}(\mathcal C)}, \label{3eq:u0C}
      \end{align}
      as  $\tau \rightarrow \infty$, using variable substitution and \eqref{3eq:psinorm}, it arrives that
      \begin{equation}\label{3eq:psiL8C}
      \left\|\psi\left(\frac{\mathbf y}{\tau}\right)\right\|_{L^8(\mathcal C)} =\tau^{\frac{3}{8}}\|\psi(\mathbf x)\|_{L^8(\mathcal C)}\leq\tau^{\frac{3}{8}}\|\psi(\mathbf x)\|_{H^{1,8}(\mathcal C)} =\mathcal O(\tau^{-\frac{1}{40}}),
            \end{equation}
            as  $\tau \rightarrow \infty$.  
      Furthermore, one  has
      \begin{equation}\label{3eq:edy}
      \|\vert \mathbf y\vert^\alpha \vert e^{\rho \cdot \mathbf x}\vert \|_{L^{\frac{8}{7}}(\mathcal C)}
      =\left( \int_{\mathcal C}\vert \mathbf y\vert ^{\frac{8}{7}\alpha}e^{-\frac{8}{7}\mathbf d\cdot \mathbf y}\mathrm d\mathbf y\right)^{\frac{7}{8}}
      \leq\left( \int_{\mathcal C}\vert \mathbf y\vert ^{\frac{8}{7}\alpha}e^{-\frac{8}{7}\zeta\vert \mathbf y\vert}\mathrm d\mathbf y \right)^\frac{7}{8}\leq \frac{C}{\zeta^{3+\frac{8}{7}\alpha}},
      \end{equation}
      where $C=2\pi\theta_0 \Gamma(3+\frac{8}{7}\alpha)(\frac{7}{8})^{3+\frac{8}{7}\alpha}$. Hence, $\|\vert \mathbf y\vert^\alpha \vert e^{\rho \cdot \mathbf x}\vert \|_{L^{\frac{8}{7}}(\mathcal C)}$ is a positive constant which only depends on $\theta_0$, $\zeta$ and $\alpha$.
      Combining  \eqref{3eq:psiL8C}, \eqref{3eq:u0C} and \eqref{3eq:I22} with \eqref{3eq:xalpha}, one has \eqref{3eq:u0alpha}.

      Furthermore, we have
      	\begin{equation}\label{3eq:u0alhagam1}
      	\int_{\Gamma_h}\vert \mathbf x\vert ^\alpha \vert u_0\vert \mathrm d \sigma
      	\leq \int_{\Gamma_h}\vert \mathbf x\vert \vert e^{\rho \cdot \mathbf x}\vert \mathrm d\sigma
      	+\int_{\Gamma_h}\vert \mathbf x\vert ^\alpha \vert e^{\rho \cdot \mathbf x}\vert\vert \psi(\mathbf x)\vert  \mathrm d\sigma,
      	\end{equation}
      	and we can easily get \eqref{3eq:xalphagam} by using polar coordinates transformation and Proposition \ref{1prop:gamma},
      	\begin{equation}\label{3eq:xalphagam}
      	\int_{\Gamma_h}\vert \mathbf x\vert \vert e^{\rho \cdot \mathbf x}\vert \mathrm d\sigma \lesssim \frac{1}{\tau^{\alpha +2}}+\frac{1}{\tau}e^{-\frac{1}{2}\zeta h\tau},
      	\end{equation}
	as  $\tau \rightarrow \infty$.  
      	Then letting $\mathbf y=\tau\mathbf x$ and  utilizing H\"older ineuqality, it can be obtained that
      	\begin{align}
      	\int_{\Gamma_h}\vert \mathbf x\vert ^\alpha\vert e^{\rho \cdot \mathbf x}\vert \vert \psi(\mathbf x)\vert \mathrm d\sigma
      	&\leq 
      	\frac{1}{\tau^{\alpha +2}}\int_{\partial \mathcal C}\vert \mathbf y \vert ^\alpha \vert e^{-\mathbf d \cdot \mathbf y}\vert \left \vert \psi \left(\frac{\mathbf y}{\tau}\right)\right \vert \mathrm d \sigma \notag \\
      	&\leq \frac{1}{\tau^{\alpha +2}} \|\vert \mathbf y \vert ^\alpha \vert e^{-\mathbf d\cdot \mathbf y}\vert \|_{L^{\frac{8}{7}}(\partial \mathcal C)}\left\| \psi \left(\frac{\mathbf y}{\tau}\right)\right\|_{L^8(\partial \mathcal C)}. \label{3eq:xalphagam1}
      	\end{align}
Similar to \eqref{3eq:edy}, we know that $\|\vert \mathbf y \vert ^\alpha \vert e^{-\mathbf d\cdot \mathbf y}\vert \|_{L^{\frac{8}{7}}(\partial \mathcal C)}$ is a positive constant. By virtue of variable substitution, trace theorem and \eqref{3eq:psinorm}, it arrives that
      	\begin{equation}\label{3eq:psigam}
      	\left\| \psi \left(\frac{\mathbf y}{\tau}\right)\right\|_{L^8(\partial \mathcal C)}
      	\lesssim \tau^{\frac{1}{4}}\|\psi(\mathbf x)\|_{H^{\frac{7}{8},8}(\partial \mathcal C)}\lesssim \tau^{\frac{1}{4}}\|\psi(\mathbf x)\|_{H^{1,8}( \mathcal C)}\lesssim \tau^{-\frac{3}{20}}, 
      	\end{equation}
	as  $\tau \rightarrow \infty$.  
      	Combining \eqref{3eq:xalphagam}, \eqref{3eq:xalphagam1} and \eqref{3eq:psigam} with \eqref{3eq:u0alhagam1}, one has \eqref{3eq:u0alphagam}.
        	\end{proof}

      Now, we are in the position to prove Theorem \ref{3thm:v0}.
      \begin{proof}[\bf{Proof of Theorem \ref{3thm:v0}}]
      The proof of this theorem is similar to the counterpart of Theorem  \ref{thm:2D}.  Recall that $(v,w)$ is a pair of transmission eigenfunctions to \eqref{0eq:tr}.
     Using Green formula \eqref{2eq:2green} and boundary conditions in \eqref{eq:pde sys 3}, the following integral identity holds
       \begin{equation}\label{3eq:green}
       \int_{\mathcal C^h}k^2(q-1)vu_0\mathrm d\mathbf x=\int_{\Lambda _h}(w-v)\partial_\nu u_0-u_0\partial(w-v)\mathrm d\sigma-\int_{\Gamma_h}\eta u_0v\mathrm d\sigma
       \end{equation}
       where $\mathcal C^h$, $\Lambda_h$ and $\Gamma_h$ are defined by  \eqref{eq:cone2}. Let
       $$
       f_{j}=(q-1)v_j.
       $$
     Due to $q\in H^2{(\overline{\mathcal C^h})}$,  we know that $q\in C^{1/2}(\overline{ \mathcal C^h})$ by using the property of embedding of Sobolev space.
      Recall that $\eta \in C^{\alpha_1}(\overline{ \Gamma_h})$. Let $\alpha=\{ \alpha_1,1/2 \}$.  Furthermore, since the Herglotz wave function $v_j\in C^{\alpha}(\overline{ \mathcal C^h})$, it yields that $f_j\in C^\alpha(\overline{ \mathcal C^h})$. Hence one has the expansion
       \begin{equation}\label{3eq:fj}
       \begin{aligned}
       f_{j}&=f_j(\mathbf 0)+\delta f_j,\ \vert \delta f_{j}\vert\leq \| f_{j}\|_{C^{\alpha}(\overline{\mathcal C^h })}\vert \mathbf{x} \vert^{\alpha},\\
       v_{j}&=v_j(\mathbf 0)+\delta v_j,\ \vert \delta v_j\vert \leq\| v_j\|_{C^\alpha(\overline{\mathcal C^h })}\vert\mathbf x\vert ^\alpha,\\
       \eta&=\eta(\mathbf 0)+\delta \eta,\ \vert\delta\eta\vert\leq\|\eta\|_{C^\alpha(\overline{\Gamma_h})}\vert\mathbf x\vert^\alpha.
       \end{aligned}
       \end{equation}
By virtue  of (\ref{3eq:fj}), we have the following integral identityies
       \begin{equation}\label{3eq:vu0}
       k^2\int_{\mathcal C_h}(q-1)vu_0\mathrm d\mathbf x=-\sum_{m=1}^3I_m,\quad
       \int_{\Gamma_h}\eta u_0v\mathrm d\sigma=I-\sum_{m=4}^9I_m,
       \end{equation}
       where
       \begin{equation}\label{3eq:chaifen}
       \begin{aligned}
       I_1&=-k^2\int_{\mathcal C^h}(q-1)(v-v_j)u_0\mathrm d\mathbf x,\quad I_2=-\int_{\mathcal C^h}\delta f_j u_0\mathrm d\mathbf x,\\
       I_3&=-f_j(\mathbf 0)\int_{\mathcal C^h}u_0\mathrm d\mathbf x,\quad
       I_4=-\eta(\mathbf 0)\int_{\Gamma _h}(v-v_j)u_0\mathrm d \sigma\\
       I_5&=-\int_{\Gamma_h}\delta \eta(v-v_j)u_0\mathrm d\sigma,\quad I_6=-\eta(\mathbf 0)v_j(\mathbf 0)\int_{\Gamma_h}e^{\rho \cdot \mathbf x}\psi(\mathbf x)\mathrm d\sigma, \\
       I_7&=\eta(\mathbf 0)\int_{\Gamma_h}\delta v_ju_0\mathrm d\sigma,\quad I_8=-v_j(\mathbf 0)\int_{\Gamma_h}\delta \eta u_0\mathrm d\sigma,\\
       I_9&= \int_{\Gamma_h}\delta \eta \delta v_ju_0\mathrm d\sigma  ,\quad I=\eta(\mathbf 0)v_j(\mathbf 0)\int_{\Gamma_h}e^{\rho\cdot \mathbf x}\mathrm d\sigma.
             \end{aligned}
       \end{equation}
       Substituting (\ref{3eq:vu0}) into (\ref{3eq:green}), it yields that
       $$
       I=\sum_{m=1}^9I_m+J_1+J_2,
       $$
       where
       \begin{equation}\label{3eq:I4I5}
       J_1=\int_{\Lambda_{h}}u_{0}\partial_{\nu}(w-v)\mathrm d\sigma,\quad
       J_2=-\int_{\Lambda_{h}}(w-v)\partial_{\nu}u_{0}\mathrm d\sigma.
       \end{equation}
Hence,  it readily yields that
       \begin{equation}\label{3eq:I}
       \vert I\vert\leq \sum_{m=1}^9\vert I_m\vert+\vert J_1\vert+\vert J_2\vert.
       \end{equation}

       In the sequel, we derive the asymptotic estimates of $I_j$ $(j=1,\ldots,9)$ and $J_j,\ j=1,2$ with respect to the parameter $\tau$ in the CGO  solution $u_0$ when $\tau \rightarrow \infty$, separately.
       Using H\"older inequality, Proposition \ref{2prop:enorm} and (\ref{3eq:psinorm}), it is clear that
       \begin{equation}\label{3eq:I11}
       \begin{aligned}
       \vert I_1\vert
       &\leq \|v-v_j\|_{L^2(\mathcal C^h)}\|e^{\rho\cdot \mathbf x}\|_{L^2(\mathcal C^h)}+\|v-v_j\|_{L^2(\mathcal C^h)}\|e^{\rho\cdot \mathbf x}\psi(\mathbf x)\|_{L^2(\mathcal C^h)}\\
       &\lesssim j^{-\beta}\left[\left(\frac{1}{\tau^3}+\frac{1}{\tau}e^{-\zeta h\tau}\right)^{\frac{1}{2}}+\left(\frac{1}{\tau^3}+\frac{1}{\tau}e^{-2\zeta h\tau}\right)^\frac{1}{4}\tau^{-\frac{2}{5}}   \right], 
       \end{aligned}
       \end{equation}
as  $\tau \rightarrow \infty$.

    With the help of (\ref{3eq:fj}), we have
       \begin{equation}\label{3eq:I21}
       \vert I_{2} \vert \leq k^{2}\|f_{j} \|_{C^{\alpha}}\int_{\mathcal C^h}\vert \mathbf x\vert^{\alpha}\vert u_{0} \vert\mathrm d\mathbf x,
       \end{equation}
       and
       \begin{equation}\label{3eq:fj1}
       \begin{aligned}
       \|f_{j}\|_{C^{\alpha}(\mathcal C^h )}&
       \leq \|q\|_{C^{\alpha}(\mathcal C^h )}\sup_{\mathcal C^h }\vert v_{j}\vert+\|v_{j}\|_{C^{\alpha}(\mathcal C^h )}\sup_{\mathcal C^h }\vert q-1\vert.
       \end{aligned}
       \end{equation}
       Moreover, due to the property of compact embedding of H\"older spaces, one has
       \begin{equation}\label{3eq:calpha}
       \| v_{j}\|_{C^{\alpha}(\mathcal C^h )}\leq \mathrm{ diam}\left(\mathcal C^h\right)^{1-\alpha}\|v_{j}\|_{C^{1}(\mathcal C^h )},
       \end{equation}
       where diam($\mathcal C^h$) is the diameter of $\mathcal C^h$. It can be directly shown that
       \begin{equation}\label{3eq:vjc1}
       \|v_{j}\|_{C^{1}(\mathcal C^h )}\leq 4\sqrt{\pi}(1+k)\|g\|_{L^{2}(\mathbb S^{2})}.
       \end{equation}
 On the other hand, we can obtain the following estimate by using the Cauchy-Schwarz inequality,
       \begin{equation}\label{3eq:vj}
       \vert v_{j}\vert\leq 4\sqrt{\pi}\|g\|_{L^{2}(\mathbb S^{2})}.
       \end{equation}
  Using (\ref{3eq:kernel}) and $q\in C^{\alpha}(\overline {\mathcal C^h})$,
       plugging  (\ref{3eq:kernel}), (\ref{3eq:calpha}), (\ref{3eq:vjc1}) and (\ref{3eq:vj}) into (\ref{3eq:fj1}), one can arrive at
       \begin{equation}\label{3eq:fjc}
       \|f_{j}\|_{C^{\alpha}(\mathcal C^h)}\lesssim j^{\gamma},
       \end{equation}
       where $\gamma$ is a given positive constant defined in (\ref{3eq:kernel}). Substituting (\ref{3eq:u0alpha}) and (\ref{3eq:fjc}) into (\ref{3eq:I21}), we obtain
       \begin{equation}\label{3eq:I2}
       \vert I_{2} \vert \lesssim j^{\gamma}\left[\tau^{-(\alpha+\frac{121}{40})}+\left(\frac{1}{\tau^{\alpha+3}}+\frac{1}{\tau}e^{-\frac{1}{2}\zeta h\tau}\right)\right]
       \end{equation}
       as $\tau \to \infty$.

  With the help of Cauchy-Schwarz inequality and (\ref{3eq:kernel}), it yields that
       \begin{equation}\label{3eq:I31}
       \vert I_{3} \vert
       \leq \int_{\mathcal C^h} \vert e^{\rho\cdot \mathbf x}\vert \mathrm d\mathbf x+\int_{\mathcal C}\vert e^{\rho\cdot \mathbf x}\vert\vert\psi(\mathbf x)\vert\mathrm d\mathbf x,
       \end{equation}
      Similar to \eqref{3eq:edy}, we have that $\|e^{-\mathbf d\cdot \mathbf y}\|_{L^\frac{8}{7}(\mathcal C)}$ is a positive constant depending only on $\zeta$ and $\theta_0$. Letting $\mathbf y=\tau \mathbf x$ and using \eqref{3eq:psiL8C}, it can be calculated that
       \begin{equation}\label{3eq:psiC}
       \begin{aligned}
       \int_{\mathcal C}\vert e^{\rho\cdot \mathbf x}\vert\vert \psi(\mathbf x)\vert\mathrm d\mathbf x
       \leq\frac{1}{\tau^3}\|e^{-\mathbf d\cdot \mathbf y}\|_{L^{\frac{8}{7}}(\mathcal C)}\left \|\psi\left (\frac{\mathbf y}{\tau}\right )\right \|_{L^8(\mathcal C)}
       \lesssim \tau^{-\frac{121}{40}}, 
       \end{aligned}
       \end{equation}
       as  $\tau \rightarrow \infty$.  
       Therefore, with the help of Proposition \ref{2prop:enorm}, and plugging \eqref{3eq:psiC} into (\ref{3eq:I31}), one has
       \begin{equation}\label{3eq:I3}
       \vert I_{3} \vert \lesssim \tau^{-\frac{121}{40}}+\left(\frac{1}{\tau^3}+\frac{1}{\tau}e^{-\frac{1}{2}\zeta h\tau}\right),\quad \mbox{as  $\tau \rightarrow \infty$. }
       \end{equation}

        By virtue of  Cauchy-Schwarz inequality and Lemma \ref{lem:trace thm}, we can obtain that
        \begin{equation}\label{3eq:I41}
        \begin{aligned}
        \vert I_4 \vert&\lesssim \|v-v_j\|_{L^{2}(\Gamma_h)}\|e^{\rho \cdot \mathbf x}\|_{L^2(\Gamma_h)}+\|v-v_j\|_{L^2(\Gamma_h)}\|e^{\rho \cdot \mathbf x}\|_{L^4(\Gamma_h)}\|\psi(\mathbf x )\|_{L^4(\Gamma_h)}\\
        &\lesssim j^{-\beta}\left[\left(\frac{1}{\tau^2}+\frac{1}{\tau}e^{-\zeta h\tau}\right)^\frac{1}{2}+\left(\frac{1}{\tau^2}+\frac{1}{\tau}e^{-2\zeta h\tau}\right)^\frac{1}{4}\tau^{-\frac{2}{5}}\right]
        \end{aligned}
        \end{equation}
        as $\tau\to \infty$.

       With the help of Cauchy-Schwarz inequality,  Lemma \ref{lem:trace thm} and  H\"older inequality, one has
       \begin{equation}\label{3eq:I5}
       \begin{aligned}
       \vert I_5\vert
       &\lesssim\|v-v_j\|_{L^2(\Gamma_h)}(\| \vert e^{\rho\cdot \mathbf x}\vert \vert \mathbf x\vert^\alpha\|_{L^2(\Gamma_h)}+\|\vert e^{\rho\cdot \mathbf x}\vert \vert \mathbf x\vert^\alpha\|_{L^4(\Gamma_h)}\|\psi(\mathbf x)\|_{L^{4}(\Gamma_h)})\\
       &\lesssim j^{-\beta}\left[\left(\frac{1}{\tau^{(2\alpha +2)}}+\frac{1}{\tau}e^{-\zeta h\tau}\right)^{\frac{1}{2}}+\left(\frac{1}{\tau^{(4\alpha +2)}}+\frac{1}{\tau}e^{-2\zeta h\tau}\right)^{\frac{1}{4}}\tau^{-\frac{2}{5}}\right],
       \end{aligned}
       \end{equation}
as  $\tau \rightarrow \infty$.

       Similar to \eqref{2eq:psik}, it can be directly obtained that
       \begin{equation}\label{3eq:psiH8}
       \left\| \psi\left(\frac{\mathbf y}{\tau}\right)\right\|_{H^{1,8}(\mathcal C)}\leq\tau^{\frac{3}{8}}\|\psi(\mathbf x)\|_{H^{1,8}(\mathcal C)}=\mathcal O(\tau^{-\frac{1}{40}}),
       \end{equation}
       as  $\tau \rightarrow \infty$.  Therefore, following the proof of Lemma \ref{2lem:psi8} and using H\"older inequality and Lemma \ref{lem:trace thm}, we have
       	\begin{equation}\label{3eq:I61}
       	\begin{aligned}
       	\vert I_6\vert &
       	\lesssim \frac{1}{\tau^2}\|e^{-\mathbf d\cdot \mathbf y}\|_{L^{\frac{8}{7}}(\Gamma)}\left\|\psi\left(\frac{\mathbf y}{\tau}\right)\right\|_{L^8(\Gamma)}\\
       	&\lesssim \frac{1}{\tau^2} \|e^{-\mathbf d\cdot \mathbf y}\|_{L^{\frac{8}{7}}(\Gamma)} \tau^{\frac{3}{8}}\|\psi(\mathbf x)\|_{H^{1,8}(\mathcal C)}
       	\lesssim\tau^{-\frac{81}{40}}, \quad \mbox{as  $\tau \rightarrow \infty$. }
       	\end{aligned}
       	\end{equation}

 Moreover, we have the following estimates for $I_7,\ I_8$ and $I_9$ by virtue of \eqref{3eq:u0alphagam} directly,
   	\begin{align}
   	\vert I_7\vert
   	&\lesssim \|v_j\|_{C^\alpha(\Gamma_h)}\int_{\Gamma_h}\vert \mathbf x\vert^\alpha\vert u_0\vert\mathrm d\sigma
   	\lesssim j^\gamma\left[ \tau^{-(\alpha +\frac{43}{20})}+\left(\frac{1}{\tau^{\alpha +2}} +\frac{1}{\tau}e^{-\frac{1}{2}\zeta h\tau}\right) \right], \notag \\
   	\vert I_8\vert& \lesssim
   	\tau^{-(\alpha +\frac{43}{20})}+\left(\frac{1}{\tau^{\alpha +2}} +\frac{1}{\tau}e^{-\frac{1}{2}\zeta h\tau}\right),\notag \\
   %
   	\vert I_9\vert&\lesssim j^\gamma\left[ \tau^{-(2\alpha +\frac{43}{20})}+\left(\frac{1}{\tau^{2\alpha +2}} +\frac{1}{\tau}e^{-\frac{1}{2}\zeta h\tau}\right)\right],\quad \mbox{as  $\tau \rightarrow \infty$. }\label{3eq:I9}
   \end{align}

        Using Cauchy-Schwarz inequality and Lemma \ref{lem:trace thm}, we obtain that
       \begin{equation}\label{3eq:J11}
       \begin{aligned}
       \vert J_1 \vert&\leq \|u_{0}\|_{H^{\frac{1}{2}}(\Lambda_{h})} \|\partial_{\nu}(w-v)\|_{H^{-\frac{1}{2}}(\Lambda_{h})}
       \leq C\|u_{0}\|_{H^{1}(\Lambda_{h})} \|\partial_{\nu}(w-v)\|_{H^{1}(\Lambda_{h})}\\
       &\lesssim \|u_{0}\|_{H^{1}(\Lambda_{h})}
       \end{aligned}
       \end{equation}
       as $\tau\to \infty$, where $C$ is a positive constant arising from the trace theorem. By virtue of (\ref{3eq:u0lambda}) and (\ref{3eq:nablau0}), it can be calculated that
       \begin{equation}\label{3eq:J1}
       \begin{aligned}
       \vert J_1 \vert \lesssim  (1+\tau)(1+\tau^{-\frac{2}{5}})e^{-\zeta h\tau}
       \end{aligned}
       \end{equation}
       as $\tau \to \infty$, where $\zeta$ is a positive constant given in (\ref{3eq:3zeta}).
       Finally, using Cauchy-Schwarz inequality, the trace theorem and (\ref{3eq:nablau0}), we can obtain that
       \begin{equation}\label{3eq:J2}
       \begin{aligned}
       \vert J_2 \vert &\leq \|\partial _{\nu} u_0\|_{ L^{2}(\Lambda_{h})} \|w-v\|_{ L^{2}(\Lambda_{h})}
       \leq C\|\partial _{\nu} u_0\|_{ L^{2}(\Lambda_{h})} \|w-v\|_{ H^{1}(\mathcal C^h)}\\
       &\lesssim (1+\tau)(1+\tau^{-\frac{2}{5}})e^{-\zeta h\tau}, 
       \end{aligned}
       \end{equation}
as  $\tau \rightarrow \infty$. 

       Substituting (\ref{3eq:I11}), (\ref{3eq:I2}), (\ref{3eq:I3})$-$\eqref{3eq:I9}, \eqref{3eq:J1} and \eqref{3eq:J2} into (\ref{3eq:I}), we have
       \begin{align}
       \vert \eta(\mathbf 0)v_j(\mathbf 0)\vert&\left(\frac{C_{\mathcal C^h}}{\tau^2}-\frac{1}{\tau}e^{-\frac{1}{2}\zeta h \tau}\right)
       	\lesssim  j^{-\beta}\left[\left(\frac{1}{\tau^3}+\frac{1}{\tau}e^{-\zeta h\tau}\right)^{\frac{1}{2}}+\left(\frac{1}{\tau^3}+\frac{1}{\tau}e^{-4\zeta h\tau}\right)^\frac{1}{4}\tau^{-\frac{2}{5}} \right]\notag \\
       	&
       	+j^{\gamma}\left[\tau^{-(\alpha+\frac{121}{40})}+\left(\frac{1}{\tau^{\alpha+3}}+\frac{1}{\tau}e^{-\frac{1}{2}\zeta h\tau}\right)\right]\notag \\
       	&
       	+j^{-\beta}\left[\left(\frac{1}{\tau^2}+\frac{1}{\tau}e^{-\zeta h\tau}\right)^{\frac{1}{2}}+\left(\frac{1}{\tau^2}+\frac{1}{\tau}e^{-2\zeta h\tau}\right)^\frac{1}{4}\tau^{-\frac{2}{5}}\right]\notag \\
       	&
       	+j^{-\beta}\left[\left(\frac{1}{\tau^{(2\alpha +2)}}+\frac{1}{\tau}e^{-\zeta h\tau}\right)^{\frac{1}{2}}+\left(\frac{1}{\tau^{(4\alpha +2)}}+\frac{1}{\tau}e^{-2\zeta h\tau}\right)^{\frac{1}{4}}\tau^{-\frac{2}{5}}\right]\notag \\
       	&
       	+(j^\gamma+1)\left[ \tau^{-(\alpha+\frac{43}{20})}+\left(\frac{1}{\tau^{\alpha+2}}+\frac{1}{\tau}e^{-\frac{1}{2}\zeta h\tau}\right)
        \right]\notag \\
       	&
       	+j^\gamma\left[\tau^{-(2\alpha+\frac{43}{20})}+\left(\frac{1}{\tau^{2\alpha+2}}+\frac{1}{\tau}e^{-\frac{1}{2}\zeta h\tau}\right)\right]\notag \\
       	&
       	+\tau^{-\frac{81}{40}}+\tau^{-\frac{121}{40}}+\left(\frac{1}{\tau^3}+\frac{1}{\tau}e^{-\frac{1}{2}\zeta h\tau}\right)
       	+(1+\tau)(1+\tau^{-\frac{2}{5}})e^{-\zeta h\tau} \label{3D1}
       	\end{align}
       as $\tau \to \infty$, where $C_{\mathcal C^h}$ is a positive constant given in (\ref{3eq:eta}). Moreover, for sufficiently large $\tau$, we know that
       $$\frac{C_{\mathcal C^h}}{\tau^2}-\frac{1}{\tau}e^{-\frac{1}{2}\zeta h \tau}>0.$$
       Hence, multiplying $\tau^2$ on both sides of (\ref{3D1}) and taking $\tau=j^s$ and $s>0$, we derive that
       \begin{equation}\label{3D2}
       \begin{aligned}
       \left(C_{\mathcal C^h}-j^s e^{-\frac{1}{2}\zeta hj^s}\right)\vert \eta(\mathbf 0)v_j(\mathbf 0)\vert &\lesssim j^{-\beta +\frac{17}{20}s}+j^{\gamma-(\alpha+1)s}+j^{-\beta+\frac{11}{10}s}\\
       &+j^{-\beta+(-\alpha+\frac{11}{10})s}
       +j^{\gamma-\alpha s}+j^{\gamma-2\alpha s}
       \end{aligned}
       \end{equation}
       as $\tau\to \infty$. Recalling that $\gamma/\alpha<\frac{10}{11}\beta$, we can choose $s\in (\gamma/\alpha ,\frac{10}{11}\beta )$. Hence in (\ref{3D2}), by letting $j\to \infty$, we  prove that
       $$\lim_{j\to \infty}\vert\eta(\mathbf 0)v_j(\mathbf 0)\vert=0.$$
       Since $\eta(\mathbf 0)\not=0$, we have $\lim_{j\to \infty }\vert v_j(\mathbf 0) \vert=0$. Using (\ref{2eq:herg}) and integral mean value theorem, we can obtain (\ref{3eq:delv0}).

       The proof is complete.
      \end{proof}

\subsection{Proof of Theorem \ref{3:cubiod}}\label{subsec:42}

	In order to prove Theorem \ref{3:cubiod}, we first give a crucial estimate in  the following proposition. It is pointed out that  $\mathcal K$ is a cuboid cone in this subsection, where $\mathbf 0$ is the apex of $\mathcal  K$. Denote  ${\sf  cone}(\mathbf a, \mathbf b)=\{\mathbf x\in \mathbb R^3~|~\mathbf x=c_1\mathbf a+c_2\mathbf b,\, \forall c_i\geq 0,\  i=1,2\} $, where $\mathbf a$ and $\mathbf b$ are  fixed vectors.  Let $\mathbf e_1=(1,0,0)^\top$, $\mathbf e_2=(0,1,0)^\top$ and $\mathbf e_3=(0,0,1)^\top$.  Suppose that the faces $\partial \mathcal K=\cup_{i=1}^3\partial \mathcal K_i $, where  $ \mathcal K_1={\sf  cone}(\mathbf e_1, \mathbf e_3) $, $ \mathcal K_1={\sf  cone}(\mathbf e_1, \mathbf e_2) $  and $ \mathcal K_1={\sf  cone}(\mathbf e_2, \mathbf e_3) $.
	
	\begin{prop}\label{pro:43}
		Let $\mathbf d=(1,1,1)^\top $ and $\mathbf d^\perp=(1,-1,0)^\top $. Denote $z_j=\rho_1 \cdot \hat{\mathbf x}_j(\theta_\xi )$, where
		\begin{align}\label{eq:x_j}
			\hat{\mathbf x}_1(\theta_\xi )=\begin{bmatrix}
			0\\ \sin \theta_\xi \\ \cos \theta_\xi
		\end{bmatrix},\quad \hat{\mathbf x}_2(\theta_\xi )=\begin{bmatrix}
			\sin \theta_\xi\\ 0 \\ \cos \theta_\xi
		\end{bmatrix},\quad \hat{\mathbf x}_3(\theta_\xi )=\begin{bmatrix}
			  \cos \theta_\xi \\ \sin \theta_\xi \\ 0
		\end{bmatrix}
		\end{align}
		with a fixed $\theta_\xi \in (0,\pi/2)$, and $\rho_1=\mathbf d +\mathrm i \mathbf d^\perp$. It holds that
		\begin{align}\label{eq:bound z}
				\left|	\sum_{j=1}^3 \frac{1}{z_j^2} \right| \geq \frac{\sin^3\theta_\xi }{30} >0.
		\end{align}
	\end{prop}
	\begin{proof}
		By direct calculations, we have
		\begin{equation}\label{eq:456 theta}
			\sum_{j=1}^3 \frac{1}{z_j^2}=\frac{S(\theta_\xi )}{z_1^2},\quad S(\theta_\xi )=1+\left(\frac{z_1}{\bar z_1}\right)^2 +c_1 z_4,
		\end{equation}
		where
		\begin{align*}
			c_1=&\frac{|z_1|^4}{||z_1|^2-\sin \theta_\xi  \cos \theta_\xi +\mathrm i \cos \theta_\xi (\cos\theta_\xi+\sin \theta_\xi  )|^4},\\
			z_4=&\left(|z_1|^2-1/2\sin2\theta_\xi -\mathbf i \cos \theta_\xi (\cos\theta_\xi +\sin\theta_\xi  ) \right)^2.
		\end{align*}
	By noting $\theta_\xi \in (0,\pi/2)$, it yields that $c_1\geq 0.05\sin\theta_\xi$ and $\Re(z_4) \geq 2 \sin^2 \theta_\xi $. Hence according to \eqref{eq:456 theta}, we obtain \eqref{eq:bound z}.	
			\end{proof}
	
	\begin{prop}\label{pro:44}
		Assume that $\mathcal K^h$ is a truncated cuboid. Let $\Gamma _h=\partial \mathcal K^h\cap B_h$ and $\rho$ be defined in \eqref{1eq:eta}, where $\mathbf d=(1,1,1)^\top $ and $\mathbf d^\perp=(1,-1,0)^\top $. Then one has
		\begin{equation}\label{eq:estcuboid}
		\left \vert \int_{\Gamma_h}e^{\rho\cdot \mathbf x}\mathrm d \sigma\right \vert \geq \frac{C^{'}_{\mathcal K^h}}{\tau^2}-\mathcal O(\frac{1}{\tau}e^{-\frac{1}{2}\zeta h \tau}),
		\end{equation}
		for sufficiently large $\tau$, where $C^{'}_{\mathcal K ^h}$ is a positive number not depending on $\tau$.
	\end{prop}
	
	\begin{proof}
		Since  $\mathcal K$ is a cuboid, by the geometrical setup and notations in this subsection, we have $\Gamma_h=\Gamma_{h1}\cup \Gamma_{h2}\cup\Gamma_{h3}$, where $\Gamma _{h1}:=\partial \mathcal K_1\cap B_h,\ \Gamma _{h2}:=\partial \mathcal K_2\cap B _h,\ \Gamma _{h3}:=\partial \mathcal K_3\cap B_h$.
		
		According to Proposition \ref{1prop:gamma}, it can be derived that
		\begin{equation}
		\begin{aligned}\notag
		\int_{\Gamma_h}e^{\rho \cdot \mathbf x} \mathrm d \sigma
		&=\int_{\Gamma _{h1}}e^{\rho \cdot \mathbf x} \mathrm d \sigma+\int_{\Gamma _{h2}}e^{\rho \cdot \mathbf x} \mathrm d \sigma+\int_{\Gamma _{h3}}e^{\rho \cdot \mathbf x} \mathrm d \sigma\\
		&= \frac{1}{2\pi \tau^2} \sum_{j=1}^3  \frac{1}{(\rho_1\cdot \hat{\mathbf x}_j(\theta_\xi))^2} -\mathcal O(\frac{1}{\tau}e^{-\frac{1}{2}\zeta h \tau}),
		\end{aligned}
		\end{equation}
		where $\theta_\xi \in (0,\pi/2)$ is fixed.
		By virtue of Proposition \ref{pro:43}, we complete the proof.
	\end{proof}
	
	\begin{proof}[The proof of Theorem \ref{3:cubiod}]
		Using the fact that $f=(q-1)v\in C^\alpha(\overline{\mathcal K^h}),\ v\in C^\alpha(\overline{\mathcal K^h}) ,\eta\in C^\alpha(\overline {\Gamma _h}) $, we have the following expansion
		\begin{equation}\label{3eq:exalpha}
		\begin{aligned}
		f&=f(\mathbf 0)+\delta f,\quad \vert \delta f \vert \leq \|f\|_{C^\alpha}\vert \mathbf x \vert ^\alpha,\\
		v&=v(\mathbf 0)+\delta v,\quad \vert \delta v \vert \leq \|v\|_{C^\alpha}\vert \mathbf x \vert ^\alpha,\\
		\eta&=\eta(\mathbf 0)+\delta \eta, \quad \vert \delta \eta\vert \leq \|\eta\|_{C^\alpha}\vert \mathbf x\vert ^\alpha.
		\end{aligned}
		\end{equation}
	Combining the integral identity \eqref{3eq:green}  with \eqref{3eq:exalpha},  it arrives that
		\begin{equation}\label{3eq:cubiod1}
		\eta(\mathbf 0)v(\mathbf 0)\int_{\Gamma_h}e^{\rho \cdot \mathbf x}\mathrm d \sigma=\sum_{i=1}^6 I_i+J_1+J_2,
		\end{equation}
		where
		\begin{equation}\label{3eq:chaicub}
		\begin{aligned}
		I_1&=f(\mathbf 0)\int_{\mathcal K^h}u_0\mathrm d\mathbf x,\quad I_2= \int_{\mathcal K^h}\delta f u_0\mathrm d\mathbf x,\quad
		I_3=\eta(\mathbf 0)v(\mathbf 0)\int_{\Gamma _h}\psi(\mathbf x)e^{\rho \cdot \mathbf x}\mathrm d\sigma,\\
		I_4&=\eta (\mathbf 0)\int_{\Gamma _h}\delta v u_0\mathrm d\sigma,\quad
		I_5=v(\mathbf 0)\int_{\Gamma _h}\delta \eta u_0\mathrm d\sigma,\quad I_6=\int_{\Gamma _h}\delta \eta\delta v u_0\mathrm d\sigma,\\
		J_1&=-\int_{\Lambda_h}(w-v)\partial_{\nu}u_0\mathrm d\sigma,\quad J_2=\int_{\Lambda_h}u_0\partial_{\nu}(w-v)\mathrm d\sigma.
		\end{aligned}
		\end{equation}
	There must exist a convex conic cone $\mathcal C$ contains the cuboid cone $\mathcal K$, namley $\mathcal K \subset \mathcal C$. Hence, by virtue of \eqref{3eq:I3} and \eqref{3eq:u0alphagam}, we have
		\begin{equation}\label{I1'}
		\vert I_1\vert
		\leq\vert f(\mathbf 0)\vert \int_{\mathcal K^h}\vert u_0\vert\mathrm d\mathbf x
		\leq  \vert f(\mathbf 0)\vert \int_{\mathcal C^h}\vert u_0\vert\mathrm d\mathbf x \lesssim \tau^{-\frac{121}{40}}+\left(\frac{1}{\tau^3}+\frac{1}{\tau}e^{-\frac{1}{2}\zeta h\tau}\right),
		\end{equation}
		and
		\begin{equation}\label{I2'}
		\vert I_2\vert
		\leq\int_{\mathcal K^h} \vert \delta f u_0\vert \mathrm d\mathbf x
		\leq \int_{\mathcal C^h} \vert \delta f u_0\vert \mathrm d\mathbf x
		\lesssim \tau^{-(\alpha +\frac{43}{20})}+\left(\frac{1}{\tau^{\alpha +2}} +\frac{1}{\tau}e^{-\frac{1}{2}\zeta h\tau}\right),
		\end{equation}
as  $\tau \rightarrow \infty$.  
		
	In view of \eqref{3eq:I61}, we have
		\begin{equation}\label{I3'}
		\vert I_3\vert \lesssim \tau^{-\frac{81}{40}}.
		\end{equation}
		
		In addition, by using \eqref{3eq:u0alphagam} in Lemma \ref{3lem:u0est}, we have the following inequalities:
		\begin{align}\label{I4'}
		\vert I_4\vert &\lesssim \tau^{-(\alpha +\frac{43}{20})}+\left(\frac{1}{\tau^{\alpha +2}} +\frac{1}{\tau}e^{-\frac{1}{2}\zeta h\tau}\right),\\
		\vert I_5 \vert &\lesssim \tau^{-(\alpha +\frac{43}{20})}+\left(\frac{1}{\tau^{\alpha +2}} +\frac{1}{\tau}e^{-\frac{1}{2}\zeta h\tau}\right), \label{I5'} \\
		\vert I_6\vert &\lesssim \tau^{-(2\alpha +\frac{43}{20})}+\left(\frac{1}{\tau^{2\alpha +2}} +\frac{1}{\tau}e^{-\frac{1}{2}\zeta h\tau}\right),  \label{I6'}
		\end{align}
		as  $\tau \rightarrow \infty$. Moreover, by using \eqref{3eq:J1} and \eqref{3eq:J2}, we have
		\begin{equation}\label{J1'}
		\vert J_1 \vert \lesssim  (1+\tau)(1+\tau^{-\frac{2}{5}})e^{-\zeta h\tau}
		\end{equation}
		and
		\begin{equation}\label{J2'}
		\vert J_2 \vert \lesssim  (1+\tau)(1+\tau^{-\frac{2}{5}})e^{-\zeta h\tau}.
		\end{equation}
		as  $\tau \rightarrow \infty$.  
Let $\rho$ be defined in  \eqref{1eq:eta} with $\mathbf d=(1,1,1)^\top $ and $\mathbf d^\perp=(1,-1,0)^\top $. By Proposition \ref{pro:44}, one has \eqref{eq:estcuboid}.   Plugging  \eqref{I1'}-\eqref{J2'} and \eqref{eq:estcuboid} into \eqref{3eq:cubiod1}, it arrives that
		\begin{equation}\label{I'}
		\begin{aligned}
		\vert \eta(\mathbf 0) v(\mathbf 0)\vert &\left(\frac{C^{'}_{\mathcal K^h}}{\tau^2}-\frac{1}{\tau}e^{-\frac{1}{2}\zeta h \tau}\right)
		\lesssim \tau^{-\frac{121}{40}}+\left(\frac{1}{\tau^3}+\frac{1}{\tau}e^{-\frac{1}{2}\zeta h\tau}\right)
		+\tau^{-(\alpha +\frac{43}{20})}\\
		&+\left(\frac{1}{\tau^{\alpha +2}} +\frac{1}{\tau}e^{-\frac{1}{2}\zeta h\tau}\right)
		+\tau^{-\frac{81}{40}}
		+\tau^{-(\alpha +\frac{43}{20})}+\left(\frac{1}{\tau^{\alpha +2}}
		+\frac{1}{\tau}e^{-\frac{1}{2}\zeta h\tau}\right)\\
		&+\tau^{-(2\alpha +\frac{43}{20})}+\left(\frac{1}{\tau^{2\alpha +2}} +\frac{1}{\tau}e^{-\frac{1}{2}\zeta h\tau}\right)\\
		&+(1+\tau)(1+\tau^{-\frac{2}{5}})e^{-\zeta h\tau},
		\end{aligned}
		\end{equation}
		where the positive constant  $C^{'}_{\mathcal K^h}$ not depending on $\tau$ is defined in \eqref{eq:estcuboid}.
Multiplying $\tau^2$ on both sides of \eqref{I'} and letting $\tau \to \infty $, one has
		$$\vert \eta (\mathbf 0)v(\mathbf 0)\vert=0.$$
Due to $\eta (\mathbf 0)\not =0$, we complete the proof of Theorem \ref{3:cubiod}.
	\end{proof}

\subsection{Proof of Corollary \ref{cor2}}\label{subsec:43}

Due to the proof of Corollary \ref{cor2} (b) can be obtained by adopting the similar process as one of Corollary \ref{cor2} (a), hence we only give the proof of Corollary \ref{cor2} (a). Firstly, we give the following proposition.
\begin{prop}\cite[Lemma\ 2.4]{DFLY}\label{3prop:eta3}
	Let $\mathcal C^{h}$ and $\rho$ be defined in (\ref{eq:cone2}) and (\ref{1eq:eta}), respectively. Then we have
	\begin{equation}\label{3eq:eta3}
	\left\vert \int_{\mathcal C^{h}} e^{\rho \cdot \mathbf x}\mathrm d\mathbf x\right\vert \geq \frac{\widetilde{C_{\mathcal C^h}}}{\tau^3}-\mathcal O\left(\frac{1}{\tau}e^{-\frac{1}{2}\zeta h \tau}\right),
	\end{equation}
	for sufficiently large $\tau$, where $\widetilde{C_{\mathcal C^h}}$ is a positive number only depending on the opening angle $\theta_0$ of $\mathcal C$ and $\zeta$.
\end{prop}

\begin{proof}[\bf{The proof of Corollary \ref{cor2}}(a)]
	The following integral identity can be obtained according to \eqref{3eq:green}:
	\begin{equation}
	k^2f_j(\mathbf 0)\int_{\mathcal C^h}e^{\rho \cdot \mathbf x}=I_1+I_2+I_3+J_1+J_2,
	\end{equation}
	where $I_m,\ m=1,2,3$, $J_1$ and $J_2$ defined in  \eqref{3eq:chaifen}.
	
	With the help of \eqref{3eq:I11}, \eqref{3eq:I2}, \eqref{3eq:I3}, \eqref{3eq:J1} and Proposition \ref{3prop:eta3}, we have the following integral inequality
	\begin{equation}\label{3eq:fj0}
	\begin{aligned}
	k^2\left[\frac{\widetilde{C_{\mathcal C^h}}}{\tau^3}-\frac{1}{\tau}e^{-\frac{1}{2}\zeta h\tau}\right]&\vert f_j(\mathbf 0)\vert \lesssim j^{-\beta}\left[\left(\frac{1}{\tau^3}+\frac{1}{\tau}e^{-\zeta h\tau}\right)^{\frac{1}{2}}+\left(\frac{1}{\tau^3}+\frac{1}{\tau}e^{-2\zeta h\tau}\right)^\frac{1}{4}\tau^{-\frac{2}{5}}   \right]\\
	&+j^{\gamma}\left[\tau^{-(\alpha+\frac{121}{40})}+\left(\frac{1}{\tau^{\alpha+3}}+\frac{1}{\tau}e^{-\frac{1}{2}\zeta h\tau}\right)\right]\\
	&+\tau^{-\frac{121}{40}}+(1+\tau)(1+\tau^{-\frac{2}{5}})e^{-\zeta h\tau}
	\end{aligned}
	\end{equation}
	as $\tau \to \infty$. For sufficiently large $\tau$, we know that
	$$ \frac{\widetilde{C_{\mathcal C^h}}}{\tau^3}-\frac{1}{\tau}e^{-\frac{1}{2}\zeta h\tau}>0.$$
	Then multiplying $\tau^3$ in the both sides of \eqref{3eq:fj0} and letting $\tau \to \infty$ and $\tau =j^s$, one has
	\begin{equation}\label{3eq:fjalpha}
	k^2\widetilde{C_{\mathcal C^h}}\vert f_j(\mathbf 0)\vert \lesssim j^{-\beta+\frac{37}{20}s}+j^{\gamma-\alpha s}.
	\end{equation}
	 Due to the assumption that $\gamma <\frac{20}{37}\alpha\beta$, we choose $s\in (\gamma/\alpha, \frac{20}{37}\beta)$. By letting $j\to \infty$, we have
	$$\vert f_j(\mathbf 0)\vert =0.$$
	Since $q(\mathbf 0)\neq 1$, the proof of this corollary is complete.
\end{proof}



\section{Visibility and unique recovery results for the inverse scattering problem}\label{sec:inverse}

In this section, we show that when a medium scatter with a conductive transmission  boundary condition possesses either one of a convex planar corner, a convex polyhedral corner, or a convex conic corner, it  radiates a non-trivial far field pattern, namely, the visibility of this scatterer occurs. Furthermore, when the medium scatter is visible, it can be uniquely determined by a single far field measurement under generic physical scenarios.


In the following theorem, it indicates that a conductive medium possesses an aforementioned corner under generic physical conditions always scatters.

\begin{thm}\label{thm:medium conical scat}
	Consider the conductive medium scattering problems \eqref{eq:contr}. Let $(\Omega ; q,\eta )$   be the medium scatterer  associated with \eqref{eq:contr}, where $\Omega$ is a bounded Lipschitz domain with a connected complement in $\mathbb R^n$, $n=2,3$.  If either of the following conditions is fulfilled, namely,
	\begin{itemize}
\item[(a)] when $\Omega\Subset \mathbb R^2$, there exists a sufficient small $h\in \mathbb R_+$ such that $\Omega\cap B_h=S_h$, where $S_h$ is defined by \eqref{1eq:sec}, $q\in  H^2(\overline{S_h} )$, $\eta \in C^\alpha(\overline{\Gamma_h^\pm  } )$ satisfying  $\alpha\in  (0,1)$ and $\eta(\mathbf 0)  \neq 0$, and  $\
\Gamma_h^\pm =\partial S_h\setminus \partial B_h $;
\item[(b)] when $\Omega\Subset \mathbb R^3$,  there exists a sufficient small $h\in \mathbb R_+$ such that $\Omega\cap B_h=\mathcal K^h$, where $\mathcal K^h$ is a cuboid  defined by \eqref{eq:cone3}, $q\in  H^2(\overline{\mathcal K^h} )$,   $\eta \in C^\alpha(\overline{\Gamma_h  } )$ satisfying  $\alpha\in  (0,1)$ and $\eta(\mathbf 0)  \neq 0$, and  $\
\Gamma_h =\partial \mathcal K^h\setminus \partial B_h $;

\item[(c)] when $\Omega\Subset \mathbb R^3$, there exists a sufficient small $h\in \mathbb R_+$ such that $\Omega\cap B_h=\mathcal K^h$, where $\mathcal K^h$ is a polyhedral corner but not a cuboid, then $q \in H^2(\overline{\mathcal K^h})$ satisfying  $q(\mathbf 0)\not =1$  and $\eta \equiv 0$ on $ \partial {\mathcal K^h}\setminus \partial B_h$;

\item[(d)] when $\Omega\Subset \mathbb R^3$,  there exists a sufficient small $h\in \mathbb R_+$ such that $\Omega\cap B_h=\mathcal C^h$, where $\mathcal C^h$ is defined by \eqref{eq:cone1}, $q\in  H^2(\overline{\mathcal C^h} )$,   $\eta \in C^\alpha(\overline{\Gamma_h  } )$ satisfying  $\alpha\in  (0,1)$,  $\eta(\mathbf 0)  \neq 0$, and  $\
\Gamma_h =\partial \mathcal C^h\setminus \partial B_h $;

	\end{itemize}
 then $\Omega$ always scatters for any incident wave satisfying \eqref{eq:incident}.
	\end{thm}

\begin{proof}
	By contradiction, suppose that the mediums scatterer $\Omega$ possesses either one of a convex planar corner, a convex polyhedral corner, and a convex conic corner, where the assumptions (a)-(d) is fulfilled. Assume that $\Omega$ is non-radiating, namely, the far field pattern $u^\infty \equiv 0$. By virtue of Rellich lemma, the total wave field $u$ and incident wave $u^i$ satisfies \eqref{0eq:tr} associated with the incident wave number $k$. It is clear that the incident $u^i$ is $\alpha$-H\"older continuous and non vanishing near the underlying corner. According to Corollaries \ref{cor1} and \ref{thm:poly}, Theorems \ref{thm:3alpha} and  \ref{3:cubiod}, one has $u^i$ must vanish at the corresponding corner  point, where we get the contradiction.
	
	The  proof is complete.
\end{proof}


In the following, we shall study the unique recovery for the inverse problem \eqref{5eq:ineta}  associated with the  conductive scattering problem \eqref{eq:contr} in $\mathbb  R^3$. In the field of inverse scattering problems, it is concerned  with the shape determination  of $\Omega$ by a minimum far-field measurement (cf. \cite{DR2018}). We utilize the local geometrical characterization of transmission eigenfunctions near a corner in  Section \ref{sec:3D} to establish the uniqueness regarding the shape determination of \eqref{5eq:ineta} by a single measurement under generic physical scenario, where a single far-field measurement means that the underlying  far-field pattern is generated only by a single incident wave $u^i$.   The unique determination results of \eqref{5eq:ineta}  for recovering the material parameters associated with \eqref{eq:contr} by  infinitely many far-field measurements  with a fixed  frequency can  be  found in \cite{OX,BHK,HK2020}. We obtain local unique recovery results for the determination of $\Omega$ without a-prior knowledge on the material parameters $q$ and $\eta$ in this section. When $\Omega$ is a cuboid or a corona shape scatterer with a conductive transmission boundary condition, the corresponding global uniqueness results on the shape determination can be drawn under generic physical scenarios. It is pointed out that when $\eta\equiv 0$ on $\partial \Omega$, namely consider the inverse problem \eqref{5eq:ineta} associated with the corresponding scattering problem 
   \begin{equation}\label{polyeq:contr}
   \begin{cases}
   &\Delta u^-+k^2qu^-=0  \hspace*{3.1cm} \mbox{in}\quad \Omega,\\
   &\Delta u^++k^2u^+=0   \hspace*{3.3cm} \mbox{in}\quad \mathbb R^n\setminus \Omega,\\
   &u^+=u^-,\quad \partial_\nu u^+=\partial _\nu u^-, \hspace*{1.7cm} \mbox{on} \quad\partial \Omega,\\
   &u^+=u^i+u^s,\hspace*{3.8cm} \mbox{in}\quad \mathbb R^n,\\
   &\lim_{r\to \infty}r^{(n-1)/2}(\partial _ru^s-iku^s)=0,\hspace*{0.4cm} r=\vert \mathbf x\vert,
   \end{cases}
   \end{equation}	
we can establish global unique recovery results for the shape of $\Omega$ within convex polyhedral or corona shape geometries by a single far-field  measurement, whereas the corresponding single-measurement uniqueness result regarding the shape determination of a convex polygon or cuboid  medium associated with \eqref{polyeq:contr} was studied in \cite{HU&SALO}.

In Theorem \ref{thm:u10}, we show the local uniqueness results for \eqref{5eq:ineta}, which aims to recover  a  scatterer $(\Omega;q,\eta)$ by knowledge of the far-field pattern $u^\infty(\hat{\mathbf x};u^i)$ with a single measurement. First, let us introduce  the admissible class of the conductive scatterer and the related notations in our study.

\begin{defn}\label{def:admiss}
	Let $\Omega$ be a bounded Lipschitz domain in $\mathbb R^3$ with a connected complement and $(\Omega;k,\mathbf d,q,\eta)$ be a conductive scatterer with the incident plane wave $u^i=e^{ik\mathbf x\cdot \mathbf d}$, where $\mathbf d\in \mathbb S^2$ and $k\in \mathbb R_+$. Consider the scattering problem \eqref{eq:contr}. Denote $u$  by the total wave field, which is associated with \eqref{eq:contr}. The scatterer $\Omega$ is said to be admissible if the following conditions are fulfilled:
	\begin{itemize}
		\item[(a)] $q\in L^\infty (\Omega)$ and $\eta\in L^\infty(\partial \Omega)$.
		\item [(b)] After rigid motions, we assume that $\mathbf 0\in \partial \Omega$. Recall that $\mathcal C^h$ and $\mathcal K^h$  are defined in \eqref{eq:cone2} and \eqref{eq:kr0} respectively, where $\mathbf 0$ is the apex of the conic corner $\mathcal C^h$  or the convex polyhedral corner $\mathcal K^h$. If $\Omega$ possesses a  convex conic corner $\mathcal C^h$ (or a cuboid corner $\mathcal K^h$), then $q \in H^2(\overline{\mathcal C^h})$ (or $q \in H^2(\overline{\mathcal K^h})$)  and $\eta\in C^\alpha (\overline {\Gamma_h})$ satisfying $\eta(\mathbf 0)\not =0$  and  $\alpha\in (0,1)$, where  $\Gamma_h= \mathcal C^h\cap \partial \Omega$ (or  $\Gamma_h= \mathcal K^h\cap \partial \Omega$). If $\Omega$ possesses a convex polyhedral corner $\mathcal K^h=B_h\cap \Omega$, then $q \in H^2(\overline{\mathcal K^h})$ satisfying  $q(\mathbf 0)\not =1$  and $\eta \equiv 0$ on $ \overline{\mathcal K^h}\cap \partial \Omega$.
		\item [(c)] The total wave field $u$ is non-vanishing everywhere in the sense that for any $\mathbf x\in \mathbb R^3$,
		\begin{equation}\label{eq:assum 51}
		\lim_{\lambda  \to +0}\frac{1}{m(B(\mathbf x,\rho))}\int_{B(\mathbf x,\lambda )}\vert u(\mathbf x)\vert\mathrm d\mathbf x\not=0,
		\end{equation}
		where $m(B(\mathbf x,\lambda ))$ is the measure of $B(\mathbf x,\lambda )$.
	\end{itemize}
\end{defn}

\begin{rem}\label{rem:51}
	The assumption \eqref{eq:assum 51} is a technical condition for deriving the uniqueness results, which can be fulfilled under generic physical scenarios. For example, when $k\cdot {\rm diam}(\Omega )\ll 1 $, by the well-posedness of the direct scattering problem \eqref{eq:contr} (cf. \cite[Theorem 2.4]{OX}), the condition \eqref{eq:assum 51} can be satisfied. The detailed discussion on this point can be found in \cite[Page 44]{DCL}. We believe that \eqref{eq:assum 51} can be  fulfilled under other physical settings, where we choose not to explore this aspect in this paper and shall investigate it in the future.
\end{rem}

\begin{thm}\label{thm:u10}
	Consider the  conductive scattering problem \eqref{eq:contr}  with  two conductive scatterers  $(\Omega_j;k,\mathbf d,q_j,\eta_j),j=1,2,$ in $  \mathbb R^3$. 
	Let $u^\infty_{j}(\hat{\mathbf {x}};u^i)$ be the far-field pattern associated with the scatterers $(\Omega_j;k,\mathbf d,q_j,\eta_j),j=1,2$ and the incident field $u^i$. If $(\Omega_j;k,\mathbf d, q_j,\eta_j)$ are admissible and
\begin{equation}\label{eq:thm51 cond}
	u^\infty_1(\hat{\mathbf x};u^i)=u^\infty_2(\hat{\mathbf x};u^i)
	\end{equation}
	for all $\hat{\mathbf x}\in \mathbb S^2$  with a fixed incident $u^i$. Then
	\begin{equation}\label{eq:set min}
	\Omega_1\Delta \Omega_2:=(\Omega_1\setminus \Omega_2)\cup(\Omega_2\setminus\Omega_1)
	\end{equation}
	cannot contain a convex conic corner or a cuboid corner. Furthermore, if $\Omega_1$ and $\Omega_2$ are two cuboids, then  $\Omega_1=\Omega_2$.
\end{thm}


\begin{proof}
 	We prove this theorem by contradiction. Suppose that $\Omega_1\Delta \Omega_2$ contains a convex conic  corner. Without loss of generality, we assume that the underlying convex conic corner $\mathcal C^h \subset {\Omega_2}\setminus {\Omega_1 } $, where  $\mathbf 0\in \partial \Omega_2$ and $\Omega_2\cap B_h=\mathcal C^h$ with a sufficient small $h\in \mathbb R_+$ such that $B_h\subset \mathbb R^3 \setminus \overline{\Omega_1}$.

 Due to \eqref{eq:thm51 cond}, with the help of Rellich's Theorem (cf.\cite{DR}), it holds that $u_1^s=u_2^s$ in $\mathbb R^3\setminus(\overline{\Omega}_1\cup \overline{\Omega}_2)$, we have
 	\begin{equation}\label{eq:u1=u2}
 u_1(\mathbf x)=u_2(\mathbf x),\  \forall  \mathbf x \in \mathbb R^3\setminus(\overline{\Omega}_1\cup \overline{\Omega}_2).
 \end{equation}
 Since $\Gamma_h=\partial \mathcal C^h\cap \partial \Omega_2$,  by  virtue of transmission conditions on $\partial \Omega_2$ of \eqref{eq:contr} and \eqref{eq:u1=u2},  it yields that
 \begin{equation}\label{eq:55 trans}
 u_2^+=u_2^-=u_1^+,\  \partial u_2^-=\partial u_2^++\eta_2 u_2^+=\partial u_1^+ +\eta_2 u_1^+\  \mathrm {on} \ \Gamma_h.
 \end{equation}
 According to \eqref{eq:55 trans} and direct scattering problems \eqref{eq:contr} associated with  $(\Omega_j;k,\mathbf d,q_j,\eta_j)$,   one has
 \begin{equation}\notag 
   \begin{cases}
   &\Delta u_2^-+k^{2}q_2u_2^-=0  \quad \hspace*{2.1cm} \mbox {$\mathrm {in}\  \mathcal C^h$},\\
   &\Delta u_1^++k^{2}u_1^+=0   \quad \hspace*{2.5cm}\mbox {$\mathrm {in}\ \mathcal C^h$},\\
   &u_2^-=u_1^+,\ \partial_{\nu}u_2^-=\partial_{\nu}u_1^++\eta_2 u_1^+ \hspace*{0.5cm}\mbox{$\mathrm {on}\   \Gamma_h$}.
   \end{cases}
   \end{equation}
 By the well-posdeness of the  direct scattering problem \eqref{eq:contr}, it yields that $u_2^-\in H^1( \mathcal C^h)$ and $u_1^+$ is real analytic in $B_h$. By virtue of the condition (b) in Definition \ref{def:admiss}, using Theorem \ref{thm:3alpha}, we know that $u_1(\mathbf 0)=0$,
  which is contradicted  to the admissibility condition (c) in Definition \ref{def:admiss}.

  The first conclusion of  this theorem concerning a cuboid corner can be proved similarly by using Theorem \ref{3:cubiod}. We omit the proof.

By the convexity of two cuboids $\Omega_1$ and $\Omega_2$ and the first conclusion  of this theorem, it is ready to know that  $\Omega_1=\Omega_2$.

 The proof is complete.
 \end{proof}



In the following we introduce an admissible class $\mathcal T$ of corona shape, which shall be used in Theorem \ref{thm:finiteconic}.
\begin{defn}\label{def:cone3}
	Let $D$ be a convex bounded Lipschitz domain with a connected complement $\mathbb R^3\setminus\overline{D}$. If there exit finite many strictly convex conic cones $\mathcal C_{\mathbf x_j,\theta_j}(j=1,2,\dots,\ell,\ell\in \mathbb N)$ defined in \eqref{eq:cone1} such that
	\begin{itemize}
		\item [(a)]
		the apex $\mathbf x_j\in \mathbb R^3\setminus \overline{D}$ and  let $\mathcal C_{\mathbf x_j,\theta_j}^{\ast}=\mathcal C_{\mathbf x_j,\theta_j}\setminus \overline D$ respectively, where the apex $\mathbf x_j$ belongs to the strictly convex bounded conic corner of $\mathcal C_{\mathbf x_j,\theta_j}^{\ast}$;
		\item [(b)]
		$\partial \overline{\mathcal C_{\mathbf x_j,\theta_j}^{\ast}}\setminus \partial \overline{\mathcal C_{\mathbf x_j,\theta_j}}\subset \partial \overline D$ and $\cap _{j=1}^\ell\partial \overline{\mathcal C_{\mathbf x_j,\theta_j}^{\ast}}\setminus \partial \overline{\mathcal C_{\mathbf x_j,\theta_j}}=\emptyset$;
    \item[(c)] $\Omega :=\cup_{j=1}^\ell\mathcal C_{\mathbf x_j,\theta_j} \cup D$ is admissible described by Definition \ref{def:admiss};
	\end{itemize}
	then $\Omega$ is said to belong to an admissible class $\mathcal T$ of corona shape.
\end{defn}


A global unique recovery for the admissible scatter belonging to $\mathcal T$ of corona shape is shown in Theorem \ref{thm:finiteconic}, which can be proved by using Theorem \ref{thm:u10} and the assumptions in Theorem \ref{thm:finiteconic}. Indeed,  the assumptions \eqref{eq:ass 56a} and \eqref{eq:ass 56b} imply that the set difference of two scatters $\Omega_1$ and $\Omega_2$ cannot contain a convex  conic corner if $\Omega_j \in \mathcal T$, $j=1,2$.

\begin{thm}\label{thm:finiteconic}
	Suppose that $\Omega_{m},m=1,2$  belong to the admissible class $\mathcal T$ of corona shape, where
	$$
	\Omega_m=\cup_{j^{(m)}=1}^{\ell^{(m)}}\mathcal C_{\mathbf x_{j^{(m)}},\theta_{j^{(m)}}}\cup D_m,\quad m=1,2.
	$$
	Consider the conductive scattering problem \eqref{eq:contr} associated with the admissible conductive scatterers  $\Omega_{m},m=1,2$. Let $u_j^\infty(\hat{\mathbf x};u^i)$ be the far-field pattern associated with the scatterers $(\Omega_m;\mathcal C_{\mathbf x_{j^m},\theta_{j^m}}),m=1,2$ and the incident field $u^i$. If the following conditions:
	\begin{subequations}
		\begin{align}
			 D_1&=D_2, \label{eq:ass 56a} \\
			 \theta_{i^{(1)}}&=\theta_{j^{(2)}} \   \mbox{for} \  i^{(1)} \in \{1,\ldots,\ell^{(1)}\}\  \mbox{and}  \  j^{(2)}\in \{1,\ldots,\ell^{(2)}
		\}\ \mbox{when} \ \mathbf x_{i^{(1)}}=\mathbf x_{j^{(2)}},  \label{eq:ass 56b}
		\end{align}
	\end{subequations}
	and \eqref{eq:thm51 cond} are satisfied,
	then $\ell^{(1)}=\ell^{(2)},\ \mathbf x_{j^{(1)}}=\mathbf x_{j^{(2)}}$ and $\theta_{j^{(1)}}=\theta_{j^{(2)}}$, where $j^{(m)}=1,\dots \ell^{(m)}$, $m=1,2$. Namely, one has $\Omega_1=\Omega_2$. 	
\end{thm}


In Theorem \ref{thm:polycon}, we first show a local uniqueness result regarding a polyhedral corner  by a single measurement, where we can prove this theorem in a similar manner  as for Theorem \ref{thm:u10} by utilizing Corollary \ref{thm:poly}. Hence the detailed proof of Theorem \ref{thm:polycon} is omitted. we emphasize that  an admissible  convex polyhedral scatterer $\Omega$ can be uniquely determined by a single far-field measurement, which a global uniqueness result for \eqref{5eq:ineta} associated with \eqref{eq:contr} is established.

\begin{thm}\label{thm:polycon}
   Consider the conductive scattering problem \eqref{eq:contr} with conductive scatterers $(\Omega_j;k,\mathbf d,q_j,\eta_j),j=1,2,$ in  $\mathbb R^3$.
   Let $u^\infty_{j}(\hat{\mathbf {x}};u^i)$ be the far-field pattern associated with the scatterers $(\Omega_j;k,\mathbf d,q_j,\eta_j),j=1,2$ and the incident field $u^i$. If $(\Omega_j;k,\mathbf d,q_j,\eta_j)$ are admissible and \eqref{eq:thm51 cond} is fulfilled, then
   $ \Omega_1\Delta\Omega_2$ defined by \eqref{eq:set min}
   cannot contains a convex polyhedral corner. Furthermore, if $\Omega_1$ and $\Omega_2$ are two admissible convex polyhedrons, then
  $$\Omega_1=\Omega_2.$$
\end{thm}

	Consider the direct scattering problem \eqref{polyeq:contr} associated with a convex polyhedron medium $(\Omega;k,\mathbf d,q)$, which is a special case of \eqref{eq:contr} by letting $\eta \equiv 0$ on $\partial \Omega$. In Corollary \ref{cor:53 noeta}, we give a global unique determination for a convex polyhedron $\Omega$ associated with the direct scattering problem \eqref{polyeq:contr}  by a single far-field measurement under generic physical settings. Corollary \ref{cor:53 noeta} can be proved directly by using Theorem \ref{thm:polycon} and the detailed proof is omitted. Compared with the corresponding uniqueness result in \cite{HU&SALO} for the shape determination of a  cuboid scatterer by a  single measurement,  we relax the geometrical restriction on the uniqueness determination regarding medium shapes  by a single measurement  from a cuboid to  a  general convex polyhedron.
	
	
	\begin{cor}\label{cor:53 noeta}
		Consider the  scattering problem \eqref{polyeq:contr} with  scatterers $(\Omega_j;k,\mathbf d,q_j),j=1,2,$ in  $\mathbb R^3$.
   Let $u^\infty_{j}(\hat{\mathbf {x}};u^i)$ be the far-field pattern associated with the scatterers $(\Omega_j;k,\mathbf d,q_j),j=1,2$ and the incident field $u^i$. Assume that the total wave field $u_j$ corresponding to  \eqref{polyeq:contr}  associated with $(\Omega_j;k,\mathbf d,q_j)$ $(j=1,2)$ satisfies \eqref{eq:assum 51}. Suppose that  $\Omega_j$ is a convex polyhedron, $j=1,2$. Denote $\mathcal V(\Omega_j )$ by a set composed by all vertexes of $\Omega_j$ with  $j=1,2$. For any $\mathbf x_{c,j} \in \mathcal V(\Omega_j )$, if there exists sufficient small $h\in\mathbb R_+$ such that $q_j\in H^2(\overline{ \mathcal  K_{\mathbf x_{c,j} }^h } )$ with $q_j(\mathbf x_{c,j} ) \neq 1$ for $j=1,2$,  where $\mathcal  K_{\mathbf x_{c,j} }^h=\Omega \cap  B_h( \mathbf x_{c,j})\Subset \Omega_j$ , then the condition \eqref{eq:thm51 cond} implies that $\Omega_1=\Omega_2.$
	\end{cor}

When  the shape  of an admissible  scatter $\Omega$ is uniquely determined by a single measurement, under a-prior knowledge the potential $q$ associated with $\Omega$ we can recover the surface parameter $\eta$ by  a single measurement   provided that $\eta$ is a non-zero constant. We can  use a similar argument for proving \cite[Theorem 4.2]{DCL} to establish  Theorem \ref{eta0}. The detailed  proof is omitted. The technical condition \eqref{eq:assum k} can be easily fulfilled under generic physical scenarios; see the detailed discussion in \cite[Remark 4.2]{DCL}.


\begin{thm}\label{eta0}
	Consider the conductive scattering problem \eqref{eq:contr} with the admissible conductive scatterers $(\Omega_m;k,\mathbf d,q,\eta_m)$ in $\mathbb R^3$, where  $\eta_m\not =0$, $m=1,2$,  are two constants.  Let $u_m^\infty(\hat{\mathbf x};u^i)$ be the far-field pattern with the scatterers $(\Omega_m;k,\mathbf d,q,\eta_m),m=1,2$ and the incident field $u^i$. Suppose that
	$$u_1^\infty(\hat{\mathbf x};u^i)=u_2^\infty(\hat{\mathbf x};u^i),\ for\ all\ \hat{\mathbf x}\in \mathbb S^2$$
	with a fixed incident wave $u^i$. If
	\begin{equation}\label{eq:assum k}
		\mbox{ $k$ is not an eigenvalue of the partial differential operator $\Delta +k^2q$, }
	\end{equation}
	and $\Omega_m$ is a cuboid ($m=1,2$), we have $\eta_1=\eta _2$. Similarly, when
	$$
	\Omega_m=\cup_{j^{(m)}=1}^{\ell^{(m)}}\mathcal C_{\mathbf x_{j^{(m)}},\theta_{j^{(m)}}}\cup D_m \in \mathcal T,\quad m=1,2,
	$$
if  the conditions \eqref{eq:assum k},  \eqref{eq:ass 56a} and \eqref{eq:ass 56b} are fulfilled, one has $\eta_1=\eta _2$.

\end{thm}

\section*{Acknowledgements}
The work of H. Diao is supported by a startup fund from Jilin University and NSFC/RGC Joint Research Grant No. 12161160314. The work of H. Liu is supported by the Hong Kong RGC General Research Funds (projects 12302919, 12301420 and 11300821) and the NSFC/RGC Joint Research Fund (project N\_CityU101/21).

	\end{document}